\def\myfontsize{10pt}
\IfFileExists{atevdegs.conf}{\input{atevdegs.conf}}{}
\documentclass[\myfontsize]{amsart}
\usepackage{amsmath, amssymb}
\usepackage{tikz}
\usetikzlibrary{calc,decorations.pathreplacing}
\usepackage[colorlinks=true, linkcolor=blue,
  citecolor=blue, urlcolor=blue]{hyperref}

\def\newthm#1#2{\newtheorem{#1}[dummy]{#2}%
  \expandafter\def\csname#2\endcsname##1{\hyperref[#1:##1]%
  {{\rm #2~\ref*{#1:##1}}}}}
\newthm{thm}{Theorem}
\newthm{lemma}{Lemma}
\newthm{prop}{Proposition}
\newthm{cor}{Corollary}
\newthm{conj}{Conjecture}
\newthm{cons}{Consequence}
\newthm{claim}{Claim}
\newthm{fact}{Fact}
\newthm{obs}{Observation}
\newthm{guess}{Guess}
\theoremstyle{definition}
\newthm{defn}{Definition}
\newthm{remark}{Remark}
\newthm{example}{Example}
\newthm{question}{Question}
\newthm{notation}{Notation}

\newcommand{\Section}[1]{\hyperref[sec:#1]{Section~\ref*{sec:#1}}}
\newcommand{\Table}[1]{\hyperref[tab:#1]{Table~\ref*{tab:#1}}}
\newcommand{\Figure}[1]{\hyperref[fig:#1]{Figure~\ref*{fig:#1}}}
\newcommand{\eqn}[1]{\hyperref[eqn:#1]{{\rm(\ref*{eqn:#1})}}}

\DeclareMathOperator{\GL}{GL}

\newcommand{\Sp}{{\mathsf{Sp}}}

\DeclareMathOperator{\Gr}{Gr}

\DeclareMathOperator{\LG}{LG}

\DeclareMathOperator{\OG}{OG}

\DeclareMathOperator{\Span}{Span}

\DeclareMathOperator{\codim}{codim}

\DeclareMathOperator{\Pic}{Pic}

\DeclareMathOperator{\coeff}{Coeff}
\DeclareMathOperator{\ord}{ord}

\DeclareMathOperator{\vdim}{vdim}
\DeclareMathOperator{\Disc}{Disc}

\newcommand{\QH}{{QH}^*}
\newcommand{\comin}{\mathrm{comin}}
\newcommand{\vir}{\mathrm{vir}}
\newcommand{\res}{\mathrm{res}}
\newcommand{\ssm}{\smallsetminus}
\newcommand{\bP}{{\mathbb P}}

\newcommand{\C}{{\mathbb C}}
\newcommand{\R}{{\mathbb R}}
\newcommand{\Q}{{\mathbb Q}}
\newcommand{\Z}{{\mathbb Z}}

\newcommand{\cB}{{\mathcal B}}

\newcommand{\cO}{{\mathcal O}}
\newcommand{\cP}{{\mathcal P}}

\newcommand{\mm}{{\mathsf m}}
\newcommand{\mA}{{\mathsf A}}
\newcommand{\mB}{{\mathsf B}}

\newcommand{\mE}{{\mathsf E}}
\newcommand{\mH}{{\mathsf H}}
\newcommand{\mO}{{\mathsf O}}
\newcommand{\mP}{{\mathsf P}}
\newcommand{\mR}{{\mathsf R}}

\newcommand{\mpB}{{\mathsf p}_B}
\newcommand{\rmline}{{\mathrm{line}}}

\newcommand{\gw}[2]{\langle #1 \rangle^{\mbox{}}_{#2}}

\newcommand{\al}{{\alpha}}
\newcommand{\be}{{\beta}}
\newcommand{\ga}{{\gamma}}

\newcommand{\la}{{\lambda}}

\newcommand{\ev}{\operatorname{ev}}

\newcommand{\wh}{\widehat}
\newcommand{\wb}{\overline}
\newcommand{\ov}{\overline}
\newcommand{\pic}[2]{\includegraphics[scale=#1]{#2}}
\newcommand{\ignore}[1]{}

\newcommand{\Mb}{\wb{\mathcal M}}

\newcommand{\noin}{\noindent}

\newcommand{\qXcb}{q^{-1}[X^{(3,2)}]}
\newcommand{\qXdb}{q^{-1}[X^{(4,2)}]}
\newcommand{\qXcc}{q^{-1}[X^{(3,3)}]}
\newcommand{\qXcba}{q^{-1}[X^{(3,2,1)}]}
\newcommand{\qXbbb}{q^{-1}[X^{(2,2,2)}]}
\newcommand{\barM}{{\overline{\mathcal{M}}}}
\newcommand{\vTev}{{\mathrm{vTev}}}

\begin{document}

\title{Tevelev degrees in Gromov-Witten theory}

\ifdefined\gitdate
\date{\gitdate\ revision {\tt \gittag}}
\else
\date{December 29, 2021}
\fi

\author{Anders~S.~Buch}
\address{Department of Mathematics, Rutgers University, 110
  Frelinghuysen Road, Piscataway, NJ 08854, USA}
\email{asbuch@math.rutgers.edu}

\author{Rahul Pandharipande}
\address{Departement Mathematik, ETH Z\"urich, R\"amisstrasse 101, Z\"urich
  8044, Switzerland}
\email{rahul@math.ethz.ch}

\dedicatory{In memory of Bumsig Kim}

\begin{abstract}
  For a nonsingular projective variety $X$, the virtual Tevelev degree in
  Gromov-Witten theory is defined as the virtual degree of the morphism from
  $\barM_{g,n}(X,d)$ to the product $\barM_{g,n} \times X^n$. After proving a
  simple formula for the virtual Tevelev degree in the (small) quantum
  cohomology ring of $X$ using the quantum Euler class, we provide several exact
  calculations for flag varieties and complete intersections. In the cominuscule
  case (including Grassmannians, Lagrangian Grassmannians, and maximal
  orthogonal Grassmannians), the virtual Tevelev degrees are calculated in terms
  of the eigenvalues of an associated self-adjoint linear endomorphism of the
  quantum cohomology ring. For complete intersections of low degree (compared to
  dimension), we prove a product formula. The calculation for complete
  intersections involves the primitive cohomology. Virtual Tevelev degrees are
  better behaved than arbitrary Gromov-Witten invariants, and, by recent results
  of \cite{lian.pandharipande:enumerativity}, are much more likely to be
  enumerative.
\end{abstract}



\maketitle

\setcounter{tocdepth}{1}
\tableofcontents


\section{Introduction}

\subsection{Virtual Tevelev degrees}

Let $\barM_{g,n}$ be the moduli space of Deligne-Mumford stable curves over $\C$
of genus $g$ with $n$ marked points. The moduli space $\barM_{g,n}$ is
nonsingular (as a stack), irreducible, and satisfies
\[
  \dim \barM_{g,n} = 3g-3+n \,.
\]
The stability condition implies $2g-2+n>0$.

Let $X$ be a nonsingular, projective, algebraic variety over $\C$ of dimension
$r$, and let $d \in H_2(X,\Z)$. The moduli space of stable maps
$\barM_{g,n}(X,d)$ has virtual dimension
\[
  \vdim \barM_{g,n}(X,d)= \int_d c_1(T_X) + (r-3)(1-g)+n \,,
\]
which equals the dimension of $\barM_{g,n}\times X^n$ if and only if
\begin{equation}\label{eqn:bbb}
  \int_d c_1(T_X)=r(n+g-1) \,.
\end{equation}
If the dimension constraint \eqn{bbb} holds, we expect to find a finite number
of maps from a fixed curve $(C,p_1,\ldots,p_n)$ of genus $g$ to $X$ of curve
class $d$ that send the marked points $p_i$ to fixed general points in $X$.
Tevelev degrees in Gromov-Witten theory are defined to be the corresponding
virtual count.

\begin{defn}\label{defn:mdef}%
  Let $g\geq 0$, $n\geq 0$, and $d \in H_2(X,\Z)$ satisfy $2g-2+n>0$ and the
  dimension constraint \eqn{bbb}. Let
  \[
    \tau : \barM_{g,n}(X,d)\to\barM_{g,n}\times X^n
  \]
  be the canonical morphism obtained from the domain curve and the evaluation
  maps. The \textbf{virtual Tevelev degree} $\vTev^X_{g,d,n} \in \Q$ is defined
  by
  \[
    \tau_*([\barM_{g,n}(X,d)]^\vir) =
    \vTev^X_{g,d,n} \cdot [\barM_{g,n}\times X^n]
    \in A^0(\barM_{g,n}\times X^n) \,.
  \]
  Here, $[\,]^\vir$ and $[\,]$ denote the virtual and usual fundamental classes,
  respectively. The invariant $\vTev^X_{g,d,n}$ is zero if the class $d \in
  H_2(X,\Z)$ is not effective, since then the moduli space $\Mb_{g,n}(X,d)$ is
  empty. If $g,n \geq 0$ and $d \in H_2(X,\Z)$ do not satisfy the dimension
  constraint \eqn{bbb}, we define $\vTev^X_{g,d,n}$ to vanish.
\end{defn}
\vspace{6pt}

For  $g,n,k \geq 0$ such that $2g-2+n+k > 0$, let
\[
  \tau : \barM_{g,n+k}(X,d)\to\barM_{g,n+k}\times X^n
\]
be the morphism obtained from the domain curve and evaluations at the first $n$
marked points. When the dimension constraint \eqn{bbb} holds, the more general
virtual degree $\vTev^X_{g,d,n,k} \in \Q$,
\begin{equation}\label{eqn:gvvt}
  \tau_*([\barM_{g,n+k}(X,d)]^\vir) =
  \vTev^X_{g,d,n,k} \cdot [\barM_{g,n+k}\times X^n]
  \in A^0(\barM_{g,n+k}\times X^n) \,,
\end{equation}
is proven in \Section{proof} to be independent of $k$, so
\[
   \vTev^X_{g,d,n,k}= \vTev^X_{g,d,n}
\]
in the stable case $2g-2+n > 0$. The definition of the virtual Tevelev degree
can be naturally extended to the
four{\footnote{$(g,n)=(0,0),(0,1),(0,2),(1,0)$.} unstable cases where $2g-2+n
\leq 0$ by
\[
  \vTev^X_{g,d,n}=\vTev^X_{g,d,n,k}
\]
for any sufficiently large $k$.

While general Gromov-Witten invariants of varieties can be complicated to
compute, we will see that the virtual Tevelev degrees are much better behaved.

\subsection{Quantum cohomology}

Let $\{\gamma_j\}$ be any basis\footnote{Cohomology and quantum cohomology will
always be taken here with $\Q$-coefficients.} of $H^*(X)$, and let
$\{\gamma_k^\vee\}$ be the dual basis defined by
\[
  \int_X \gamma_j \cup \gamma_k^\vee = \delta_{j,k} \,.
\]
Let $\QH(X)$ be the small quantum cohomology ring of $X$ defined via the 3-point
Gromov-Witten invariants in genus 0,
\[
  \gamma_i \star \gamma_j= \sum_{d \in H_2(X,\Z)}\sum_k \,
  \langle \gamma_i, \gamma_j, \gamma_k^\vee \rangle^X_{0,d} \,
  \cdot q^d \gamma_k \, \in \QH(X) \,,
\]
see \cite{fulton.pandharipande:notes} for an introduction. Let
\[
  \Delta = \sum_j \gamma_j^\vee \otimes \gamma_j \in H^*(X\times X)
\]
be the standard K\"unneth decomposition{\footnote{The order matters!}} of the
diagonal class.

\begin{defn}\label{defn:defE}
  The \textbf{quantum Euler class} of $X$ is
  \[
    \mE = \sum_j \gamma_j^\vee \star \gamma_j \, \in \QH(X) \,.
  \]
\end{defn}
\vspace{4pt}

The classical ($q=0$) part of the quantum Euler class is determined by the usual
topological Euler characteristic $\chi(X)$,
\[
  \mE = \chi(X) \cdot \mP\, + \, \text{$q$-corrections} \,,
\]
where $\mP\in H^{2r}(X)$ is the point class. Since $\mE$ is the image of
$\Delta$ under the canonical multiplication map
\[
  H^*(X)\otimes H^*(X) \xrightarrow{\ \star\ } \QH(X) \,,
\]
$\mE$ is independent of basis choice. The quantum Euler class was first
introduced{\footnote{The definition there differs slightly but is, in fact,
equivalent to ours.}} by Abrams in \cite{abrams:quantum}, see also \cite[Section
8]{chaput.manivel.ea:quantum*2}.

Our first result is that virtual Tevelev degrees can be computed in $\QH(X)$
using the point class and the quantum Euler class.

\begin{thm}\label{thm:qhred}%
  For $g,n \geq 0$ and $d \in H_2(X,\Z)$ we have
  \[
    \vTev^X_{g,d,n} = \coeff(\mP^{\star n} \star \mE^{\star g}, q^d \mP) \,.
  \]
\end{thm}

\noindent The notation $\mA^{\star n}$ denotes the power of
$\mA\in \QH(X)$ with respect to the quantum product. A canonical way to
write the coefficient in \Theorem{qhred} is
\[
  \coeff(\mP^{\star n} \star \mE^{\star g}, q^d \mP) =
  \int _X \coeff(\mP^{\star n}\star \mE^{\star g}, q^d) \,.
\]

Tevelev degrees have been studied in several contexts
\cite{cela.lian:generalized, cela.pandharipande.ea:tevelev, farkas.lian:linear,
lian.pandharipande:enumerativity, tevelev:scattering} starting with $X=\bP^1$.
Almost always, virtual Tevelev degrees are much better behaved than enumerative
Tevelev degrees \cite{lian.pandharipande:enumerativity} and general
Gromov-Witten invariants. Our results concern exact calculations of virtual
Tevelev degrees in two main cases: cominuscule flag varieties and low degree
complete intersections in projective spaces. An asymptotic equality between
virtual and enumerative Tevelev degrees for certain Fano varieties (including
flag varieties and low degree hypersurfaces) is proved in
\cite{lian.pandharipande:enumerativity}, so many of our calculations are actual
curve counts.

\subsection{Virtual Tevelev degrees: hypersurfaces}

Using \Theorem{qhred} and properties of the quantum cohomology of hypersurfaces,
we obtain the following results:

\vspace{8pt}
\noindent $\bullet$ %
The projective space case $X=\bP^r$ has a particularly simple answer:
 \begin{equation}\label{eqn:projspace}
  \vTev^{\bP^r}_{g,d,n}=(r+1)^g
\end{equation}
whenever the dimension constraint \eqn{bbb} is satisfied.

\vspace{8pt}
\noindent $\bullet$ The case of a quadric hypersurface $Q^r\subset \bP^{r+1}$
takes a special form. Let
\[
  \delta = \begin{cases}
    1 & \text{if $r$ is odd,} \\
    2 & \text{if $r$ is even.}
  \end{cases}
\]

\begin{thm}\label{thm:quadric}
  For nonsingular quadrics $Q^r$ of dimension $r\geq 3$,
  \[
    \vTev^{Q^r}_{g,d,n} = \frac{(2r)^g + (-1)^d (2\delta)^g}{2}
  \]
  whenever the dimension constraint \eqn{bbb} is satisfied.
\end{thm}

We index here the curve classes of $Q^r\subset \bP^{r+1}$ for $r\geq 3$ by their
associated degree $d$ in $\bP^{r+1}$. For fixed genus $g$, using the
enumerativity results of \cite{lian.pandharipande:enumerativity}, the virtual
Tevelev degrees for $Q^r$ in sufficiently high degree $d$ are actual curve
counts.

The precise statement in enumerative geometry from
\cite{lian.pandharipande:enumerativity}
is the following. Let $C$ be a fixed general curve of genus $g$. The dimension
constraint for $Q^r$ is
\[
  d = n+g-1 \,.
\]
For $d$ (and hence $n$) sufficiently large, let $(C,p_1,\ldots,p_n)$ be defined
by general points of $C$. Let $x_1,\ldots, x_n \in Q^r$ be general points. Then,
the actual count of  maps
\[
  f: C \rightarrow Q^r \ \ \ \text{satisfying} \ \ f(p_i)=x_i
\]
where $f$ has degree $d$, is equal to $\vTev^{Q^r}_{g,d,n}$. Whether the
counting problem for quadrics can be solved directly by classical techniques is
an interesting question.

The excluded $r=2$ case of the quadric surface $Q^2\subset \bP^3$ has additional
curve classes since $Q^2 \cong \bP^1 \times \bP^1$. The virtual Tevelev degrees
of $Q^2$ are determined by a product rule in \Section{prodrule}.

\vspace{8pt}
\noindent $\bullet$
For low degree hypersurfaces, we also have a complete calculation. The curve
classes of hypersurfaces of dimension at least 3  are indexed by their
associated degree in projective space.

\begin{thm}\label{thm:hypersurface}%
  Let $X_e\subset \bP^{r+1}$ be a nonsingular hypersurface of degree $e\geq 3$
  and dimension $r\geq 2e-3$.  Then, for $g+n \geq 2$ we have
  \[
    \vTev^{X_e}_{g,d,n} = ((e-1)!)^n \cdot (r+2-e)^g\cdot e^{(d-n)e-g+1}
  \]
  whenever the dimension constraint \eqn{bbb} is satisfied.
\end{thm}

Notice that for any nonsingular projective Fano variety $X$, the virtual Tevelev
degrees $\vTev^X_{g,d,n}$ for which $g+n \leq 1$ are given by
\[
  \vTev^X_{0,d,0} = 0 \,, \ \ \ \ \
  \vTev^X_{0,d,1} = \delta_{d,0} \,, \ \ \ \ \
  \vTev^X_{1,d,0} = \delta_{d,0}\, \chi(X) \,,
\]
so they can safely be left out of \Theorem{hypersurface}.

Most of the previous work on the quantum cohomology of the hypersurface
$X_e\subset \bP^{r+1}$ concerns the subalgebra of classes restricted from
$\bP^{r+1}$. However, the definition of the quantum Euler class $\mE$ involves
{\em all} of the cohomology of $X_e$. Controlling the contributions of the
primitive cohomology is perhaps the most interesting{\footnote{See also
\cite{arguz.bousseau.ea:gromov-witten, hu:big} for recent results on the
Gromov-Witten theory of complete intersections which also confront the primitive
cohomology.}} aspect of the proof of \Theorem{hypersurface}.

Our techniques apply unchanged to the case of complete intersections of low
degree. The results for complete intersections take a very similar form and are
stated in \Theorem{ccii} of \Section{secci}. Various patterns are presented
beyond our degree/dimension ranges for hypersurfaces and complete intersections
in \Section{cplint}. There are many open questions.

\subsection{Virtual Tevelev degrees: cominuscule flag varieties}

The formulas for the virtual Tevelev degrees for projective spaces $\bP^r$ and
quadrics $Q^r$ have a different flavor than for the complete intersections of
\Theorem{hypersurface}. For fixed genus $g$, the virtual Tevelev degree of
$\bP^r$ and $Q^r$ depend upon the degree $d$ at most by a parity condition (no
dependence on $d$ at all for $\bP^r$ and mod 2 dependence for $Q^r$). Since $d$
is related to $g$ and $n$ by the dimension constraint, we can also view the
dependence as periodic in $n$. We prove in \Section{flagflag} that all
cominuscule{\footnote{The full definition is reviewed in \Section{flagflag}.}
flag varieties have such simple behavior.

\begin{thm}\label{thm:ttt333}%
  Let $X$ be a cominuscule flag variety. When the dimension constraint \eqn{bbb}
  is satisfied, $\vTev^{X}_{g,d,n}$ depends only upon $g$ and
  \[
    n \mod\ \ord(\mP)\,,
  \]
  where $\ord(\mP)$ is the finite order of the point class $\mP$ in the group of
  units of the ring $\QH(X)/\langle q-1\rangle$.
\end{thm}

Cominuscule flag varieties $X$ include  Grassmannians, Lagrangian Grassmannians,
and maximal orthogonal Grassmannians. In genus 0 and 1, we find complete closed
forms for the virtual Tevelev degrees of these spaces. For general $g$, we
calculate the virtual Tevelev degrees for all cominuscule flag varieties in
terms of the eigenvalues and eigenvectors of a basic operator
\[
  [\mE/\mP]_0 : \QH(X)_{q,0} \rightarrow \QH(X)_{q,0}
\]
on the degree 0 part of the localized quantum cohomology. The formula for
virtual Tevelev degrees of arbitrary genus is presented in \Theorem{vTev_comin}
of \Section{strange}.

The operator $\mE/\mP$ is defined via quantum multiplication by $\mE$ followed
by the inverse (after localization) of quantum multiplication by the point class
$\mP$, and $[\mE/\mP]_0$ is the restriction to the degree 0 part. A crucial
property of $\mE/\mP$ is symmetry with respect to the strange duality involution
of \cite{postnikov:affine, chaput.manivel.ea:quantum}.

As an example of the effectivity of our results, the virtual Tevelev degrees of
the Grassmannian $\Gr(2,5)$, the first Grassmannian which is not a projective
space or a quadric, are
\[
  \vTev^{\Gr(2,5)}_{g,d,n} =
  \frac{5 - \sqrt{5}}{10} \left( \frac{25 + 5\,\sqrt{5}}{2} \right)^g
  + \frac{5 + \sqrt{5}}{10} \left( \frac{25 - 5\,\sqrt{5}}{2} \right)^g
\]
when the dimension constraint is satisfied. For fixed genus $g$, using the
enumerativity results of \cite{lian.pandharipande:enumerativity}, the virtual
Tevelev degrees for all flag varieties in sufficiently high degree $d$ are
actual curve counts. Several further calculations are presented in
\Section{examex} and \Section{furthercfvs}

While our formulas for virtual Tevelev degrees of cominuscule flag varieties and
complete intersections look different, they share the feature that they rely
only on a small and well-behaved subring of the quantum cohomology of $X$. For
cominuscule varieties, the calculation occurs in the subring $\QH(X)_{q,0}$. For
complete intersections, the calculation is dictated to a surprising extent by
the subring defined by classes restricted from the ambient projective space.

\Theorem{qhred} implies that the virtual Tevelev degrees of all flag varieties
are non-negative integers, and \Theorem{ccii} shows that virtual Tevelev degrees
of low-degree complete intersections are non-negative integers.
\Example{vTev_neg} shows that $\vTev^X_{1,2,1} = -64$ when $X \subset \bP^7$ is
a complete intersection of three general quadrics, so virtual Tevelev degrees
can be negative ($X$ is just out of the bounds required by \Theorem{ccii}). On
the other hand, we have not been able to find a virtual Tevelev degree that is
not an integer. It would be interesting to know if non-integer virtual Tevelev
degrees exist.

\subsection{Acknowledgments}

We thank Alessio Cela, Carl Lian, Gavril Farkas, Dhruv Ranganathan, and
Johannes Schmitt for many
discussions about Tevelev degrees, and Pierre-Emmanuel Chaput, Leonardo
Mihalcea, Nicolas Perrin, and Weihong Xu for inspiring discussions about
cominuscule quantum cohomology. A.B. was partially supported by ICERM and is
grateful for the stimulating environment provided while he participated in
ICERM's semester program on Combinatorial Algebraic Geometry in the Spring of
2021. R.P. was supported by SNF-200020-182181, ERC-2017-AdG-786580-MACI, and
SwissMAP. This project has received funding from the European Research Council
(ERC) under the European Union Horizon 2020 research and innovation program
(grant agreement No 786580).


\section{Reduction to quantum cohomology}\label{sec:proof}

\subsection{Proof of \Theorem{qhred}}

Let $X$ be a nonsingular complex projective variety of dimension $r$, and let $d
\in H_2(X,\Z)$. Let $g,n,k \geq 0$ be non-negative integers satisfying $2g-2+n+k
> 0$. If the dimension constraint \eqn{bbb} is satisfied, we have by definition
\eqn{gvvt},
\[
  \vTev^X_{g,d,n,k} \cdot [\barM_{g,n+k}\times X^n]
  = \tau_* \left([\barM_{g,n+k}(X,d)]^\vir\right)
  \in A^0(\barM_{g,n+k}\times X^n) \,,
\]
where $\tau=(\pi, \ev_{[n]})$ is the morphism
\[
  \tau : \barM_{g,n+k}(X,d)\to\barM_{g,n+k}\times X^n
\]
constructed from the factors
\[
  \pi : \barM_{g,n+k}(X,d)\to\barM_{g,n+k}
  \ \ \ \ \ \text{and} \ \ \ \ \
  \ev_{[n]} : \barM_{g,n+k}(X,d)\to X^n \,,
\]
where $\ev_{[n]}$ denotes evaluation at the first $n$ marked points.

Let $I^X_{g,d,n+k} : H^*(X)^{\otimes (n+k)} \to H^*(\barM_{g,n+k})$ denote the
Gromov-Witten class defined by
\[
  I^X_{g,d,n+k}(\al) =
  \pi_*\left( \ev_{[n+k]}^*(\al) \cap [\Mb_{g,n+k}(X,d)]^\vir \right) \,.
\]
By the projection formula, we obtain
\[
  \vTev^X_{g,d,n,k} \cdot [\barM_{g,n+k}] =
  I^X_{g,d,n+k}(\mP^{\otimes n} \otimes 1^{\otimes k}) \,,
\]
where $\mP \in H^{2r}(X)$ is the point class. By the identity insertion axiom in
Gromov-Witten theory, $\vTev^X_{g,d,n,k}$ is independent of $k$ (whenever
$2g-2+n+k>0$). Moreover, $\vTev^X_{g,d,n} = \vTev^X_{g,d,n,k}$ is given by
\begin{equation}\label{eqn:vTev}%
  \vTev^X_{g,d,n} =
  \int_{\Mb_{g,n+k}} I^X_{g,d,n+k}(\mP^{\otimes n} \otimes 1^{\otimes k})
  \cap \mP_{\Mb_{g,n+k}} \,,
\end{equation}
whether or not the dimension constraint \eqn{bbb} is satisfied. Here,
$\mP_{\Mb_{g,n+k}}$ denotes the point class in $H^*(\Mb_{g,n+k})$.

By the genus reduction axiom of Gromov-Witten theory (see
\cite{behrend:gromov-witten}) we have
\[
  \psi^* I^X_{g,d,n+k}(\mP^{\otimes n} \otimes 1^{\otimes k}) =
  I^X_{0,d,n+k+2g}(\mP^{\otimes n} \otimes 1^{\otimes k} \otimes
  \Delta^{\otimes g}) \,,
\]
where $\Delta \in H^*(X \times X)$ is the diagonal class and $\psi :
\barM_{0,n+k+2g} \to \barM_{g,n+k}$ sends the $(n+k+2g)$-pointed rational curve
$[\bP^1, p_1,p_2,\dots,p_{n+k}, z_1,z'_1,\dots,z_g,z'_g]$ to the $(n+k)$-pointed
curve with $g$ nodes obtained by identifying $z_i$ with $z'_i$ for $1 \leq i
\leq g$:
\[
\begin{tikzpicture}[x=10pt,y=10pt]
  \coordinate (o) at (0,0);
  \coordinate (c0) at (1,0);
  \coordinate (m01) at (2.4,1.5);
  \coordinate (c1) at (0,.9);
  \coordinate (m12) at (-1,1);
  \coordinate (c2) at (-.2,0);
  \coordinate (pm) at (2.5,0);
  \def\mypoint#1#2{node[circle,fill,inner sep=1.5pt,label={below:#2}] at #1 {}}
  \def\mynode{
  .. controls +(c0) and +($ (o)-(c1) $) .. ++(m01)
  .. controls +(c1) and +($ (o)-(c2) $) .. ++(m12)
  .. controls +(c2) and +($ (c1-|o) - (c1|-o) $) .. ++($ (m12|-o) - (m12-|o) $)
  .. controls +($ (c1|-o) - (c1-|o) $) and +($ (c0-|o) - (c0|-o) $)
  .. ++($ (m01|-o) - (m01-|o) $)
  }
  \draw[thick] (-.5,0) -- ++(16,0) \mynode\mynode\mynode\mynode -- ++(2.5,0)
  \mypoint{(2.5,0)}{$p_1$}
  \mypoint{(5,0)}{$p_2$}
  \mypoint{(7.5,0)}{}
  \mypoint{(10,0)}{}
  \mypoint{(12.5,0)}{$p_{n+k}$}
  node[label={below:$\cdots$}] at (8.75,0) {};
  \draw [decorate,decoration={brace,mirror,amplitude=7pt}]
  (16,-1) -- ++(10,0) node[midway,yshift=-15pt] {$g$};
\end{tikzpicture}
\]

\noindent
We therefore obtain from \eqn{vTev} that
\begin{equation}\label{eqn:v33p}
  \vTev^X_{g,d,n} = \langle \mP^{\otimes n} \otimes 1^{\otimes k} \otimes
  \Delta^{\otimes g} \rangle^X_{\odot,d,n+k+2g} \,,
\end{equation}
where, for each $m \geq 3$, we let $\langle - \rangle^X_{\odot,d,m} :
H^*(X)^{\otimes m} \to \Q$ be the map defined by
\[
  \langle \al \rangle^X_{\odot,d,m} =
  \int_{\barM_{0,m}} I^X_{0,d,m}(\al) \cap \mP_{\Mb_{0,m}} \,.
\]
The invariants $\langle \al \rangle^X_{\odot,d,m}$ should be considered
\emph{small} $m$-pointed Gromov-Witten invariants by the following Proposition
(which played a central role in Bertram's computation of the quantum cohomology
of Grassmannians \cite{bertram:quantum}).

\begin{prop}\label{prop:bertram}
  Let $m \geq 2$. Then, the $m$-fold small quantum multiplication map
  \[
    H^*(X)^{\otimes m} \to \QH(X)
  \]
  is given by
  \[
    \alpha \ \mapsto \ \sum_{d,k} \langle \al \otimes \ga_k^\vee
    \rangle^X_{\odot,d,m+1} \cdot q^d \ga_k \,.
  \]
\end{prop}
\begin{proof}
  The identity for $m=2$ is the definition of the small quantum product. We
  proceed by induction on $m$.

  Given $\al_1 \in H^*(X)^{\otimes m}$ and $\al_2 \in H^*(X)$, the splitting
  axiom of Gromov-Witten theory implies
  \[
    \langle \al_1 \otimes \al_2 \otimes \gamma_k^\vee \rangle^X_{\odot,d,m+2}
    = \sum_{d_1+d_2=d} \sum_{k'} \
    \langle \al_1 \otimes \ga_{k'}^\vee \rangle^X_{\odot,d_1,m+1} \cdot
    \langle \ga_{k'} \otimes \al_2\otimes \gamma_k^\vee
    \rangle^X_{\odot,d_2,3} \ .
  \]
  By using the induction hypothesis, we see that the map of the Proposition is
  given by iterating the small quantum product.
\end{proof}

By equation \eqn{v33p} for the virtual Tevelev degree and \Proposition{bertram},
we obtain
\[
  \vTev^X_{g,d,n}
  = \langle \mP^{\otimes n} \otimes 1^{\otimes k} \otimes \Delta^{\otimes g}
    \rangle^X_{\odot,d,n+k+2g}
  = \coeff(\mP^{\star n} \star \mE^{\star g}, q^d \mP) \,,
\]
which completes the proof of \Theorem{qhred}. \qed

\begin{example}
  We illustrate \Theorem{qhred} by computing the virtual Tevelev degrees of
  projective space $X = \bP^r$. The small quantum cohomology ring is given by
  \[
    \QH(\bP^r) = \Q[\mH,q] / \langle \mH^{r+1} - q \rangle \,,
  \]
  and we have
  \[
    \mP = \mH^r \ \ \ \ \text{and} \ \ \ \ \mE = (r+1)\, \mH^r \,.
  \]
  We identify $H_2(\bP^r,\Z)$ with $\Z$. The dimension constraint \eqn{bbb} on
  $g, d, n \geq 0$ is then
  \[
    (r + 1) d = r (n + g - 1) \,.
  \]
  When satisfied, we have
  \[
    \vTev^{\bP^r}_{g,d,n} =
    \coeff(\mP^{\star n} \star \mE^{\star g}, q^d \mP) = (r+1)^g \,.
  \]
\end{example}

\subsection{Product rule}\label{sec:prodrule}%

The virtual Tevelev degrees of a product of varieties are obtained from the
virtual Tevelev degrees of the factors as follows.

\begin{prop}\label{prop:prodrule}%
  Let $X$ and $Y$ be nonsingular projective varieties. For $g,n \geq 0$, $d_X
  \in H_2(X,\Z)$, and $d_Y \in H_2(Y,\Z)$ we have
  \[
    \vTev^{X \times Y}_{g,(d_X,d_Y),n} =
    \vTev^X_{g,d_X,n} \cdot \vTev^Y_{g,d_Y,n} \,.
  \]
\end{prop}
\begin{proof}
  The result is a direct consequence of the product rule in Gromov-Witten theory
  \cite{behrend:product}. Choose $k \geq 0$ such that $2g-2+n+k > 0$. Let $\mP_X
  \in H^*(X)$ and $\mP_Y \in H^*(Y)$ be the point classes. Then $\mP_X \otimes
  \mP_Y$ is the point class of $X \times Y$. By the product rule we have
  \[
    I^{X\times Y}_{g,(d_X,d_Y),n+k} \left(
    (\mP_X\otimes \mP_Y)^{\otimes n} \otimes 1^{\otimes k} \right) =
    I^X_{g,d_X,n+k}(\mP_X^{\otimes n} \otimes 1^{\otimes k})\cdot
    I^Y_{g,d_Y,n+k}(\mP_Y^{\otimes n} \otimes 1^{\otimes k})
    \]
    in $H^*(\barM_{g,n+k})$. The claim therefore follows from \eqn{vTev}.
\end{proof}

The dimension constraint \eqn{bbb} may be satisfied for
$\vTev^{X\times Y}_{g,(d_X,d_Y),n}$, but fail for $\vTev^X_{g,d_X,n}$ and
$\vTev^Y_{g,d_Y,n}$. Then,  all three invariants are zero.

\begin{example}
  The virtual Tevelev degrees of the quadric surface $Q^2 \subset \bP^3$ are
  easily determined by \Proposition{prodrule} since $Q^2 \cong \bP^1\times
  \bP^1$. If $2 d_1 = 2 d_2 = n+g-1$, then
  \[
    \vTev^{Q^2}_{g,(d_1,d_2),n} = 4^g \,,
  \]
  where $(d_1,d_2)$ denotes the bidegree of the curve class. In all other cases,
  the virtual Tevelev degree vanishes.
\end{example}


\section{Flag varieties}\label{sec:flagflag}

\subsection{Schubert varieties}

Let $G$ be a connected semi-simple linear algebraic group over $\C$ and fix a
maximal torus $T$ and a Borel subgroup $B$ such that $T \subset B \subset G$.
The opposite Borel subgroup $B^- \subset G$ is defined by $B \cap B^- = T$. Let
$W = N_G(T)/T$ be the Weyl group of $G$, and let $\Phi$ be the root system, with
positive roots $\Phi^+$ and simple roots $\Phi^S$. Each root $\al \in \Phi$
defines a reflection $s_\al \in W$.

A flag variety of $G$ is a projective variety $X$ with a transitive $G$-action.
Every flag variety $X$ contains a unique $B$-stable point $\mpB \in X$. We let
$P_X \subset G$ denote the parabolic subgroup stabilizing $\mpB$. We then have
the identifications
\[
  X = G/P_X = \{ g.P_X \mid g \in G \}\,, \ \  \ \mpB = 1.P_X \,.
\]

Let $W_X = N_{P_X}(T)/T$ denote the Weyl group of $P_X$, and let $W^X \subset W$
be the set of minimal length representatives of the cosets in $W/W_X$. Every
element $u \in W$ defines a $B$-stable Schubert variety $X_u = \ov{Bu.P_X}$ and
an (opposite) $B^-$-stable Schubert variety $X^u = \ov{B^-u.P_X}$. For $u \in
W^X$, we have
\[
  \dim(X_u) = \codim(X^u,X) = \ell(u) \,.
\]
The set of Schubert classes $\{[X^u] \mid u \in W^X\}$ is a basis of the
cohomology ring $H^*(X)$. The dual basis is $\{[X_u] \mid u \in W^X\}$, in the
sense that $\int_X [X_u] \cdot [X^v] = \delta_{u,v}$. Let $w_0^X$ denote the
longest element in $W^X$. Then $[X^{w_0^X}]$ is the class of a point in $X$, and
$\ell(w_0^X) = \dim(X)$. When no confusion is possible, we will also write $\mP
= [X^{w_0^X}]$ for the class of a point.

\subsection{Quantum cohomology}

Let $\QH(X)_q$ denote the localized small quantum cohomology ring of $X$. This
ring is an algebra over the Laurent polynomial ring
\[
  \Q[q^{\pm 1}] = \Q[q_\al^{\pm 1} \mid \al \in \Phi^S \ssm \Phi_X] \,,
\]
where $\Phi_X \subset \Phi$ is the root system of $P_X$.\footnote{The quantum
cohomology ring of a flag variety can also be defined with $\Z$-coefficients.}
We have $\QH(X)_q = H^*(X) \otimes_\Q \Q[q^{\pm 1}]$ as a $\Q[q^{\pm
1}]$-module, and multiplication in $\QH(X)_q$ is defined by
\[
  [X^u] \star [X^v] = \sum_{w,d} \gw{[X^u], [X^v], [X_w]}{d} \,
  q^d\, [X^w] \,,
\]
where the sum is over all $w \in W^X$ and effective degrees $d \in H_2(X,\Z)$.
Here, we use the notation $q^d = \prod q_\al^{d_\al}$ for each $d \in
H_2(X,\Z)$, where $d_\al = \int_{X_{s_\al}} d$. The invariants
$\gw{[X^u],[X^v],[X_w]}{d}$ are enumerative, in particular they are non-negative
integers (see \cite{fulton.pandharipande:notes}). By
\cite[Thm.~9.1]{fulton.woodward:quantum},  every quantum product $[X^u] \star
[X^v]$ of Schubert classes is non-zero. The quantum Euler class of a
flag variety is naturally expressed using the Schubert basis:
\begin{equation}\label{eqn:intint}
  \mE = \sum_{u \in W^X} [X_u] \star [X^u] \,.
\end{equation}

The ring $\QH(X)_q$ has the $\Q$-basis
\[
  \cB = \{ q^d [X^u] \mid d \in H_2(X) \text{ and } u \in W^X \} \,.
\]
For every $\mA \in \QH(X)$ and basis element $q^d [X^u] \in \cB$, let
\[
  \coeff(\mA, q^d [X^u])
\]
denote the coefficient of $q^d [X^u]$ when $\mA$ is expanded in the basis $\cB$.
The usual (small) quantum cohomology ring $\QH(X)$ is the subring $H^*(X)
\otimes_\Q \Q[q] \subset \QH(X)_q$ spanned by all elements $q^d [X^u]$ for which
$d \in H_2(X,\Z)$ is an effective degree.

\begin{prop}\label{prop:flagINT}%
  For any flag variety $X$, we have $\vTev^{X}_{g,d,n} \in \Z_{\geq 0}$ for all
  $g,n \geq 0$ and $d\in H_2(X,\Z)$.
\end{prop}
\begin{proof}
  Using that the quantum product $[X^u] \star [X^v]$ of any two Schubert classes
  is a non-negative integer combination of the basis $\cB$, it follows from
  \eqn{intint} that $\mP^{\star n} \star \mE^{\star g}$ is a non-negative
  integer combination of $\cB$. The claim follows since $\vTev^X_{g,d,n}$ is one
  of the coefficients in the expansion of $\mP^{\star n} \star \mE^{\star g}$ by
  \Theorem{qhred}.
\end{proof}

\subsection{Seidel classes}

Let $\QH(X)/\langle q-1 \rangle$ denote the quotient of $\QH(X)$ by the ideal
generated by $q_\al-1$ for $\al \in \Phi^S \ssm \Phi_X$.

\begin{defn}
  A Schubert class $[X^v]$ is called a \emph{Seidel class} if
  \[
    [X^v]^{\star p} = q^d
  \]
  for some integer $p \geq 1$ and $d \in H_2(X)$. A Seidel class  $[X^v]$ is an
  element of finite order in the group of units $(\QH(X)/\langle q-1
  \rangle)^\times$. Let $\ord([X^v])$ denote the order of a Seidel class.
\end{defn}

\begin{lemma}
  Let $[X^v]$ be a Seidel class of $X$. We have the following
  properties:\smallskip

  \noin{\rm(a)} $[X^v] \star [X^u] \in \cB$ for all $u \in W^X$,\smallskip

  \noin{\rm(b)} $[X_v] \star [X^v] = \mP$, the point class of $X$,\smallskip

  \noin{\em(c)} $[X^v]$ is invertible in $\QH(X)_q$ with inverse
    $[X^v]^{\star(-1)} \in \cB$,\smallskip

  \noin{\em(d)} $\coeff([X^v]\star \mA, [X^v]\star \mB) =
    \coeff(\mA, \mB)$ for all $\mA \in \QH(X)_q$ and $\mB \in \cB$.
\end{lemma}
\begin{proof}
  Let $p = \ord([X^v])$ and write $[X^v]^{\star p} = q^d$. Part (a) follows
  because any quantum product $[X^v] \star [X^w]$ is a non-zero linear
  combination of $\cB$ with non-negative integer coefficients, and $[X^v]^{\star
  p} \star [X^u] = q^d [X^u]$ is a single element of $\cB$. Part (b) follows
  from (a) because $[X_v] \cdot [X^v] = \mP$ in $H^*(X)$. Part (c) holds because
  $q^{-d}[X^v]^{\star (p-1)}$ is an inverse of $[X^v]$, and part (d) follows
  from (a) and (c).
\end{proof}

\subsection{Cominuscule flag varieties}

A simple root $\ga \in \Phi^S$ is called \emph{cominuscule} if, when the highest
root of $\Phi$ is written as a linear combination of simple roots, the
coefficient of $\ga$ is one. The non-trivial Seidel classes of $X$ correspond to
cominuscule simple roots of $G$ by some remarkable relationships.

A flag variety $Y = G/P_Y$ is called \emph{cominuscule} if $P_Y \subset G$ is a
maximal parabolic subgroup defined by excluding a cominuscule simple root,
\[
  \Phi_Y = \Phi^S \ssm \{\ga\} \,.
\]
It is proved in Bourbaki that the set
\[
  W^\comin =
  \{ w_0^Y \mid Y \text{ is a cominuscule flag variety of $G$ } \} \cup \{1\}
\]
is a subgroup of the Weyl group $W$. Furthermore, it is proved in
\cite{chaput.manivel.ea:affine} that
\[
  [X^{w_0^Y}] \star [X^u] = q^{d(Y,u)} [X^{w_0^Y u}] \ \in \QH(X)
\]
for every
$w_0^Y \in W^\comin$ and
$u \in W^X$, where the degree $d(Y,u)$ is
explicitly given in terms of root data. In particular, the group $W^\comin$ acts
on $\QH(X)/\langle q-1 \rangle$ by $w_0^Y . [X^u] = [X^{w_0^Y}] \star [X^u]$. We
state the following special case.

\begin{thm}[Chaput, Manivel, Perrin]
  \label{thm:seidel}
  Let $X$ be any flag variety of $G$, and let $Y$ be any cominuscule flag variety
  of $G$. Then, $[X^{w_0^Y}]$ is a Seidel class of $X$, and $\ord([X^{w_0^Y}])$
  is equal to the order of $w_0^Y$ in $W^\comin$.
\end{thm}

The following table shows the complete list of cominuscule flag varieties $Y$ as
well as $\ord(w_0^Y)$ in each case.\bigskip

\begin{tabular}{|c|c|c|}
  \hline
  Name & $Y$ & $\ord(w_0^Y)$ \\
  \hline
  Grassmannian of type A & $\Gr(m,N)$ & $N/\gcd(m,N)$ \\
  Lagrangian Grassmannian & $\LG(N,2N)$ & $2$ \\
  Maximal orthogonal Grassmannian & $\OG(N,2N)$ & $2$ if $N$ is even, else $4$\\
  Quadric hypersurface & $Q^r$ & $2$ \\
  Cayley Plane & $E_6/P_6$ & 3 \\
  Freudenthal variety & $E_7/P_7$ & 2 \\
  \hline
\end{tabular}
\bigskip

If $X$ is a cominuscule flag variety, then the point class $\mP \in H^*(X)$ is a
Seidel class by \Theorem{seidel}.

\begin{cor}\label{cor:vtev-modulo}
  Let $X$ be any cominuscule flag variety. The virtual Tevelev degree
  \[
    \vTev^X_{g,d,n} =
    \coeff(\mP^{\star n} \star \mE^{\star g}, q^d \mP) =
    \coeff(\mE^{\star g}, q^d \mP^{\star(1-n)})
  \]
  depends on $n$ modulo $\ord(\mP) = \ord(w_0^X)$.
\end{cor}

\begin{cor}\label{cor:vtev_zero}
  Let $X$ be any cominuscule flag variety of dimension $r$. The virtual Tevelev
  degrees of $X$ of genus 0 are determined by
  \[
    \vTev^X_{0,d,n} = \begin{cases}
      1 & \text{if $n \equiv 1$ modulo $\ord(\mP)$ and
      $\int_d c_1(T_X) = r (n - 1)$}, \\
      0 & \text{otherwise.}
    \end{cases}
  \]
\end{cor}

Some coefficients of the quantum Euler class $\mE$ of an arbitrary flag
variety have the following combinatorial description.

\begin{thm}
  \label{thm:qdiag-seidel}
  Let $X$ be any flag variety and assume that $[X_v]$ is a Seidel class. Then we
  have
  \[
    \coeff(\mE, q^d [X^v]) =
    \# \{ u \in W^X \mid [X_v] \star [X^u] = q^d [X^u] \} \,.
  \]
\end{thm}
\begin{proof}
  We have
  \[
    \begin{split}
      \coeff(\mE, q^d [X^v])
      &= \sum_{u \in W^X} \coeff([X_u] \star [X^u], q^d [X^v]) \\
      &= \sum_{u \in W^X} \gw{[X_u], [X^u], [X_v]}{d} \\
      &= \sum_{u \in W^X} \coeff([X_v] \star [X^u], q^d [X^u]) \,.
    \end{split}
  \]
  Since $[X_v]$ is a Seidel class, we have $[X_v] \star [X^u] \in \cB$, hence
  \[
    \coeff([X_v] \star [X^u], q^d [X^u]) = \begin{cases}
      1 & \text{if $[X_v] \star [X^u] = q^d [X^u]$}, \\
      0 & \text{otherwise,}
    \end{cases}
  \]
  which proves the identity.
\end{proof}

\begin{cor}
  Let $X$ be any cominuscule flag variety. The virtual Tevelev degrees of $X$ of
  genus 1 are determined by
  \[
    \vTev^X_{1,d,n} \, = \, \coeff(\mE, q^d \mP^{\star(1-n)}) \, = \,
    \# \{ u \in W^X \mid \mP^{\star n} \star [X^u] = q^d [X^u] \} \,.
  \]
\end{cor}
\begin{proof}
  The point class $\mP$ is a Seidel class of $X$ by \Theorem{seidel}. Write
  \[
    q^d \mP^{\star(1-n)} = q^e [X^v] \ \in \QH(X)_q \,.
  \]
  Then, $[X^v]$ is a Seidel class. Since $[X_v] \star [X^v] = \mP$, $[X_v]$ is
  also a Seidel class. The calculation
  \[
    q^{-d} \mP^{\star n} \star q^e[X^v]
    = q^{-d} \mP^{\star n} \star q^d \mP^{\star(1-n)} = \mP
    = q^{-e}[X_v] \star q^e[X^v]
  \]
  shows that $q^{-d}\mP^{\star n} = q^{-e}[X_v]$. We therefore obtain
  \[
    \begin{split}
      \coeff(\mP^{\star n} \star \mE, q^d \mP) &= \coeff(\mE, q^e[X^v]) \\
      &= \# \{ u \in W^X \mid [X_v] \star [X^u] = q^e [X^u] \} \\
      &= \# \{ u \in W^X \mid \mP^{\star n} \star [X^u] = q^d [X^u] \} \,,
    \end{split}
  \]
  as required.
\end{proof}

\subsection{Quadric hypersurfaces}\label{sec:quadric}

Let $X = Q^r \subset \bP^{r+1}$ be a quadric hypersurface of dimension $r \geq
3$. Then, $Q^r$ is a cominuscule flag variety. The quantum cohomology ring
$\QH(Q^r)$ is a basic example in the subject, see e.g.\
\cite{chaput.manivel.ea:quantum*1}. Structure theorems for this ring are also
special cases of results about the quantum cohomology of complete intersections
proved in \Section{cplint}. We have $\deg(q) = \deg(\mP)$, and the relation
$\mP^{\star 2} = q^2$ holds in $\QH(Q^r)$ since three general points of $Q^r$
define a plane which cuts $Q^r$ in a conic. Define the constant
\[
  \delta = \begin{cases}
    1 & \text{if $r$ is odd},\\
    2 & \text{if $r$ is even.}
  \end{cases}
\]
Then, $H^*(Q^r)$ has rank $r + \delta$. For $u \in W^X$, we have{\footnote{The
quantum multiplication here can also be deduced from the complete intersection
analysis in \Section{cplint}.}}
\[
  [X_u] \star [X^u] = \begin{cases}
    \mP & \text{if $\ell(u) \in \{ 0, r/2, r \}$,} \\
    \mP + q & \text{otherwise.}
  \end{cases}
\]
This follows by noting that $[X^u]$ is a Seidel class when $\ell(u)=r/2$ by
\Theorem{seidel}, and from the quantum Chevalley formula
\cite[Thm.~10.1]{fulton.woodward:quantum} when $\ell(u) \neq r/2$. We compute
the quantum Euler class of $Q^r$ as
\[
  \mE = \sum_{u \in W^X} [X_u] \star [X^u] = (r + \delta)\mP + (r - \delta)q \,.
\]

\begin{thm}
  Let $g,d,n \geq 0$ satisfy the constraint $d = n+g-1$. Then,
  \[
    \vTev^{Q^r}_{g,d,n} =
    \frac{(2r)^g + (-1)^d (2\delta)^g}{2} \,.
  \]
\end{thm}
\begin{proof}
  Using the binomial formula
  \[
    \mE^{\star g} = \sum_{i=0}^g \binom{g}{i} (r+\delta)^i (r-\delta)^{g-i}\,
    \mP^{\star i}\, q^{g-i}
  \]
  and the relation $\mP^{\star 2} = q^2$, we obtain
  \[
    \vTev^{Q^r}_{g,d,n}
    = \coeff(\mE^{\star g}, q^d \mP^{\star(1-n)})
    = \sum_{\substack{0 \leq i \leq g\\ n+i \text{ odd}}}
    \binom{g}{i}\, (r+\delta)^i\, (r-\delta)^{g-i} \,,
  \]
  which matches the expansion of
  \[
    \frac{\big( (r+\delta) + (r-\delta) \big)^g
    - (-1)^{n+g}\, \big( (r+\delta) - (r-\delta) \big)^g}{2} \,,
  \]
  as required.
\end{proof}

\subsection{Grassmannians of type A}

For the remaining cominuscule flag varieties we give simple formulas for the
virtual Tevelev degrees of genus 1 based on \Theorem{qdiag-seidel}. More
generally, we give formulas for all numbers $\coeff(\mE, q^e [X^v])$ for which
$[X^v]$ is a Seidel class.

Let $X = \Gr(m,N)$ be the Grassmannian of $m$-planes in $\C^N$. The dimension of
$X$ is $r = m(N-m)$. The elements of $W^X$ can be identified with partitions
$\la = (\la_1 \geq \cdots \geq \la_m \geq 0)$ for which $\la_1 \leq N-m$. The
corresponding opposite Schubert variety is given by
\[
  X^\la = \{ V \in X \mid \dim(V \cap \C^{N-m+i-\la_i}) \geq i \text{ for }
  1 \leq i \leq m \} \,,
\]
where $\C^k \subset \C^N$ denotes the $B^-$-stable subspace of dimension $k$.
The point class is $\mP = [X^{((N-m)^m)}]$, where $(a^b)$ denotes the partition
$(a,\dots,a)$ with $b$ copies of $a$.

The quantum cohomology ring $\QH(X)$ was computed in \cite{witten:verlinde,
bertram:quantum}. Elementary proofs of the facts we need can be found in
\cite{buch:quantum}. The grading of $\QH(X)$ is determined by
\[
  \deg [X^\la] = 2 |\la| = 2 \sum \la_i
\]
and $\deg(q) = 2N$. The subgroup of Seidel classes in $(\QH(X)/\langle q-1
\rangle)^\times$ is generated by the class $[X^{(1^m)}]$, the top Chern class of
the dual of the tautological subbundle on $X$. We have $\ord([X^{(1^m)}]) = N$,
and the powers of $[X^{(1^m)}]$ are given by
\[
  [X^{(1^m)}]^{\star k} = \begin{cases}
    [X^{(k^m)}] & \text{if $0 \leq k \leq N-m$,} \\
    q^{k-N+m} [X^{((N-m)^{N-k})}] & \text{if $N-m \leq k \leq N$.}
  \end{cases}
\]

\begin{thm}\label{thm:qdiag-A}%
  Let $X = \Gr(m,N)$, and let $k,d \in \Z$ satisfy $Nd + mk = r$. Then,
  \[
    \coeff(\mE, q^d [X^{(1^m)}]^{\star k}) = \binom{cN/m}{c} \,,
  \]
  where $c = \gcd(d,m)$.
\end{thm}

\begin{cor}
  Let $n,d \geq 0$ satisfy $Nd = r n$. Then,
  \[
    \vTev^{\Gr(m,N)}_{1,d,n} = \binom{cN/m}{c} \,,
  \]
  where $c = \gcd(d,m)$.
\end{cor}
\begin{proof}
  The result follows from \Theorem{qdiag-A} since
  \[
    \coeff(\mP^{\star n} \star \mE, q^d \mP)
    = \coeff(\mE, q^d \mP^{\star(1-n)})
    = \coeff(\mE, q^d [X^{(1^m)}]^{\star k}) \,,
  \]
  where $k = (N-m)(1-n)$.
\end{proof}

We will prove \Theorem{qdiag-A} using Postnikov's cylindrical model of the basis
\[
  \cB = \{q^d [X^\la] \}
\]
of $\QH(X)_q$. Define the set
\[
  \cP_X = \{ (i,j) \in \Z^2 \mid 1 \leq i \leq m
  \text{ and } 1 \leq j \leq N-m \} \,.
\]
We identify $\cP_X$ with a rectangle of boxes with $m$ rows and $N-m$ columns.
The pair $(i,j)$ represents the box in row $i$ and column $j$, where the row
number $i$ increases from top to bottom and the column number $j$ increases from
left to right. Let $\leq$ be the \emph{north-west to south-east} partial order
on $\cP_X$, defined by $(i',j') \leq (i'',j'')$ if and only if $i' \leq i''$ and
$j' \leq j''$. A partition $\la$ as above corresponds to the (lower) order ideal
in $\cP_X$ of boxes $(i,j)$ for which $1 \leq i \leq m$ and $1 \leq j \leq
\la_i$.

More generally, the elements $q^d [X^\la]$ of $\cB$ correspond to (proper,
non-empty, lower) order ideals in Postnikov's \emph{cylinder}
\cite{postnikov:affine},
\[
  \wh\cP_X = \Z^2/\Z(-m,N-m) \,.
\]
In other words, we extend the rectangle $\cP_X$ to the plane $\Z^2$, but
identify two boxes if they differ by an integer multiple of the vector
$(-m,N-m)$. The order ideal of $1 \in \cB$ is the set $I_0 = \{ (i,j) \in
\wh\cP_X \mid (i,j) \not\geq (1,1) \}$. Equivalently, $I_0$ is the set of boxes
of $\wh\cP_X \ssm \cP_X$ that are smaller than some box in $\cP_X$. The order
ideal of $[X^\la]$ is the union of $I_0$ with the order ideal of $\la$ in
$\cP_X$.

An order ideal of $\wh\cP_X$ is determined by its border, which is a path of
horizontal and vertical line segments of unit length, with the property that the
path is invariant under translation by the vector $(-m,N-m)$, see
\Figure{cylinder}. The order ideal of $q^d [X^\la]$ is obtained by translating
the border of the order ideal of $[X^\la]$ by the vector $(d,d)$. In particular,
multiplication by $q$ corresponds to translation by $(1,1)$. Similarly, quantum
multiplication by $[X^{(1^m)}]$ corresponds to translation by $(0,1)$.

\begin{figure}[h]
  \caption{The border of the order ideal of $[X^{(5,3,3,2)}]$ in the cylinder
    $\wh\cP_X$ of $X = \Gr(4,6)$. The boxes of $\cP_X$ are colored gray.}
  \label{fig:cylinder}
  \begin{center}
    \begin{tikzpicture}[x=4mm,y=4mm]
      \draw[fill=gray!30] (0,0) -- (6,0) -- (6,4) -- (0,4) -- cycle;
      \foreach \r in {1,...,3} { \draw (0,\r) -- (6,\r); }
      \foreach \c in {1,...,5} { \draw (\c,0) -- (\c,4); }
      \def\pp{-- ++(3,0) -- ++(0,1) -- ++(1,0) -- ++(0,2) -- ++(2,0) -- ++(0,1)}
      \draw [line width=0.7mm] (-7,-4) \pp\pp\pp;
      \draw [-stealth] (3,-4) -- (9,0) node [midway, xshift=7mm, yshift=-5mm]
      {$(-m,N-m)$};
      \draw [-stealth] (-6,3.5) -- ++(1,-1) node [midway, xshift=7mm] {$(1,1)$};
      \draw [-stealth] (-6,6) -- ++(1,0) node [midway, xshift=7mm] {$(0,1)$};
    \end{tikzpicture}
  \end{center}
\end{figure}

\begin{proof}[Proof of \Theorem{qdiag-A}]
  Since $[X^{(1^m)}]^{\star(N-m-k)} \star [X^{(1^m)}]^{\star k} = \mP$, it
  follows from \Theorem{qdiag-seidel} that $\coeff(\mE, q^d [X^{(1^m)}]^{\star
  k})$ is equal to the number of Schubert classes $[X^\la]$ for which
  \[
    q^{-d} [X^{(1^m)}]^{\star(N-m-k)} \star [X^\la] = [X^\la] \,.
  \]
  Since multiplication by $q^{-d}$ corresponds to translation of paths by the
  vector $(-d,-d)$, and multiplication by $[X^{(1^m)}]^{\star(N-m-k)}$
  corresponds to translation by $(0,N-m-k)$, we deduce that $\coeff(\mE, q^d
  [X^{(1^m)}]^{\star k})$ is equal to the number of paths in the plane that go
  through the upper-right corner of $\cP_X$ and are invariant under translation
  by both of the vectors $(-d,N-m-k-d)$ and $(-m,N-m)$. Notice that the
  constraint $Nd + mk = m(N-m)$ says that these vectors are parallel. We
  therefore must count the number of paths through the upper-right corner of
  $\cP_X$ that are invariant under translation by the greatest common divisor of
  the two vectors, which is the vector $(-c,(N-m)c/m)$, where $c = \gcd(d,m)$.
  Since such a path is determined by the first $cN/m$ steps, and exactly $c$ of
  these steps must be vertical, there are $\binom{cN/m}{c}$ such paths.
\end{proof}

\subsection{Lagrangian Grassmannians}

Let $X = \LG(N,2N)$ be the Lagrangian Grassmannian of maximal isotropic
subspaces in a symplectic vector space of dimension $2N$ over $\C$. The quantum
cohomology ring $\QH(X)$ was computed in \cite{kresch.tamvakis:quantum}, and
elementary proofs can be found in \cite{buch.kresch.ea:gromov-witten}. We have
$\dim(X) = \frac{1}{2}N(N+1)$, $\deg(q) = 2N+2$, and $\mP^{\star 2} = q^N$ in
$\QH(X)$. Since the only Seidel classes of $X$ are $1$ and $\mP$, all
coefficients of $\mE$ at Seidel classes are given by the following result.

\begin{thm}
  Let $n,d \geq 0$ satisfy $2d = nN$. Then,
  \[
    \vTev^{\LG(N,2N)}_{1,d,n} = \begin{cases}
      \coeff(\mE, \mP) = 2^N & \text{if $n$ is even;} \\
      \coeff(\mE, q^{N/2}) = 2^{N/2} & \text{if $n$ is odd.}
    \end{cases}
  \]
\end{thm}

\begin{figure}[h]
  \caption{The partially ordered set $\wh\cP_X$ for $X = \LG(6,12)$, with the
    boxes of $\cP_X$ colored gray.}
  \label{fig:lgcyl}
  \begin{center}
    \begin{tikzpicture}[x=4mm,y=4mm]
      \def\zz{-- ++(0,-1) -- ++(1,0)}
      \draw [fill=gray!30] (6,0) -- (0,0) \zz\zz\zz\zz\zz\zz -- cycle;
      \draw (-2,2) \zz\zz\zz\zz\zz\zz\zz\zz\zz -- ++(0,-1);
      \draw (5,2) \zz\zz\zz\zz\zz\zz\zz\zz\zz -- ++(0,-1);
      \foreach \i in {1,...,5} \draw (\i,0) -- (\i,-\i) -- (6,-\i);
    \end{tikzpicture}
  \end{center}
\end{figure}
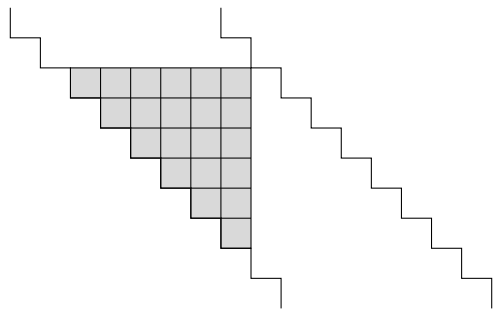

We use a generalization of Postnikov's cylinder constructed in
\cite{buch.chaput.ea:positivity}. The Schubert classes of $\LG(N,2N)$ correspond
to order ideals in a triangular set of boxes $\cP_X = \{ (i,j) \in \Z^2 \mid 1
\leq i \leq j \leq N \}$, where the row number $i$ increases from top to bottom,
and the column number $j$ increases from left to right. The partial order is the
north-west to south-east order, defined by $(i',j') \leq (i'',j'')$ if and only
if $i' \leq i''$ and $j' \leq j''$. The elements of $\cB$ correspond to
(non-empty, proper, lower) order ideals in the larger set $\wh\cP_X = \{ (i,j)
\in \Z^2 \mid i \leq j \leq i+N \}$, see \Figure{lgcyl}. Multiplication by $q$
moves the border of an order ideal one unit diagonally in south-east direction.
Given an order ideal $\la \subset \cP_X$, the order ideal of $\mP \star [X^\la]$
is obtained by reflecting $\la$ in a diagonal and attaching it to the right side
of $\cP_X$, see \Figure{lgpoint}.

\begin{proof}
  If $n$ is even, the formula states that $H^*(X)$ has rank $2^N$. Assume that
  $n$ is odd. Then $N$ must be even, and $\coeff(\mE, q^{N/2})$ is the number of
  order ideals $\la \subset \cP_X$ for which $q^{-N/2} \mP \star [X^\la] =
  [X^\la]$. This holds if and only if the border of $\la$ is a path from the
  upper-right corner of $\cP_X$ to the middle point on the south-west side of
  $\cP_X$, and this path must be symmetric under reflections in the diagonal,
  see \Figure{lgsym}. The path is therefore determined by the first $N/2$ steps,
  so there are $2^{N/2}$ such paths.
\end{proof}

\begin{figure}[h]
  \caption{The order ideal of $\mP \star [X^\la]$ is obtained from the order
    ideal of $[X^\la]$ as follows, for $\la = (4,1)$.}
  \label{fig:lgpoint}
  \begin{center}
    \begin{tikzpicture}[x=4mm,y=4mm]
      \def\zz{-- ++(0,-1) -- ++(1,0)}
      \draw (0,0) \zz\zz\zz\zz\zz\zz -- ++(0,6) -- cycle;
      \draw [line width=.7mm] (2,-2) -- (2,-1) -- (4,-1) -- (4,0) -- (6,0);
      \draw (8,2) \zz\zz\zz\zz\zz\zz\zz\zz\zz -- ++(0,-1);
      \draw (15,2) \zz\zz\zz\zz\zz\zz\zz\zz\zz -- ++(0,-1);
      \draw (10,0) -- ++(6,0) -- ++(0,-6);
      \draw [line width=.7mm] (16,-6) -- ++(0,2) -- ++(1,0) -- ++(0,2) -- ++(1,0);
    \end{tikzpicture}
  \end{center}
\end{figure}
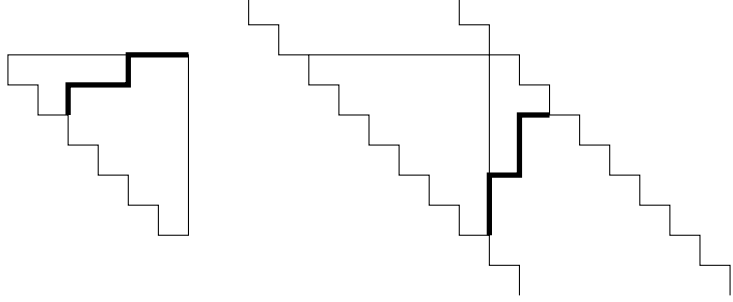

\begin{figure}[h]
  \caption{The border of the order ideal of $[X^{(4,2,1)}]$ is symmetric under
    reflection in the diagonal. Equivalently, $q^{-N/2} \mP \star
    [X^{(4,2,1)}] = [X^{(4,2,1)}]$ in $\QH(X)$.}
  \label{fig:lgsym}
  \begin{center}
    \begin{tikzpicture}[x=4mm,y=4mm]
      \def\zz{-- ++(0,-1) -- ++(1,0)}
      \draw (0,0) \zz\zz\zz\zz\zz\zz -- ++(0,6) -- cycle;
      \draw [line width=.7mm] (3,-3) -- (3,-1) -- (4,-1) -- (4,0) -- (6,0);
      \draw (8,2) \zz\zz\zz\zz\zz\zz\zz\zz\zz -- ++(0,-1);
      \draw (15,2) \zz\zz\zz\zz\zz\zz\zz\zz\zz -- ++(0,-1);
      \draw (10,0) -- ++(6,0) -- ++(0,-6);
      \draw [line width=.7mm,gray]
            (13,-3) -- ++(0,2) -- ++(1,0) -- ++(0,1) -- ++(2,0);
      \draw [line width=.7mm]
            (16,-6) -- ++(0,2) -- ++(1,0) -- ++(0,1) -- ++(2,0);
    \end{tikzpicture}
  \end{center}
\end{figure}
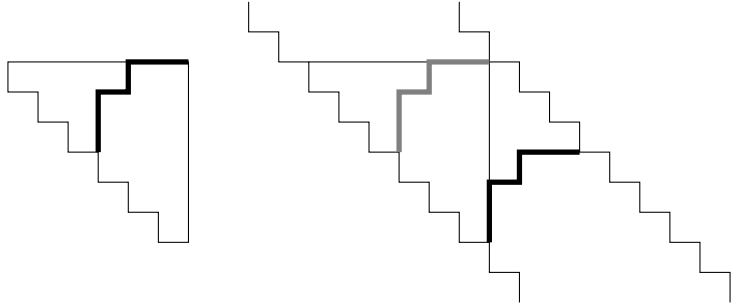

\subsection{Maximal orthogonal Grassmannians}

Let $X = \OG(N,2N)$ be the Grassmannian parametrizing (one component of) maximal
isotropic subspaces in an orthogonal complex vector space of dimension $2N$. The
quantum cohomology ring $\QH(X)$ was computed in
\cite{kresch.tamvakis:quantum*1}, elementary proofs can be found in
\cite{buch.kresch.ea:gromov-witten}. We have $\dim(X) = \frac{1}{2}N(N-1)$ and
$\deg(q) = 4N-4$. The Seidel classes{\footnote{We use a quotient $\mA/\mB$ of
cohomology classes only if $\mB$ is invertible in the localized quantum
cohomology ring, in which case it should be interpreted as $\mA/\mB = \mA \star
\mB^{\star(-1)}$.}} of $X$ are $1$, $\mP$, $[X^{N-1}]$, and $\mP / [X^{N-1}]$.
Here $X^{N-1} \cong \OG(N-1,2N-2)$ denotes the Schubert variety of maximal
isotropic subspaces that contain a fixed isotropic vector. We have
$[X^{N-1}]^{\star 2} = q$ and $\mP^{\star 4} = q^N$. If $N$ is even, we
furthermore have $\mP^{\star 2} = q^{N/2}$. Let $\mE_{q=1}$ denote the image of
$\mE$ in $\QH(X)/\langle q-1\rangle$.

\begin{thm}
  Let $X = \OG(N,2N)$. We have
  \[
    \begin{split}
      \coeff(\mE_{q=1}, \mP) &= 2^{N-1}, \\
      \coeff(\mE_{q=1}, 1) &= \begin{cases}
        \coeff(\mE, q^{N/4}) = 2^{N/2} & \text{if $N \equiv 0$ (mod 4),} \\
        0 & \text{otherwise,}
      \end{cases} \\
      \coeff(\mE_{q=1}, [X^{N-1}]) &= \begin{cases}
        \coeff(\mE, q^{\frac{N-2}{4}} [X^{N-1}]) = 2^{N/2} &
        \text{if $N \equiv 2$ (mod 4),} \\
        0 & \text{otherwise, and}
      \end{cases} \\
      \coeff(\mE_{q=1}, \mP/[X^{N-1}]) & = 0 \,.
    \end{split}
  \]
\end{thm}
\begin{proof}
  The coefficient of $\mP$ is the topological Euler characteristic of $X$, and
  the coefficient of $\mP/[X^{N-1}]$ is zero since $\deg\, [X^{N-1}]$ is not a
  multiple of $\deg(q)$. The two remaining cases are equivalent to the identity
  \[
    \coeff(\mE, [X^{N-1}]^{\star N/2}) = 2^{N/2}
  \]
  when $N$ is even. The elements of $W^X$ can be identified with order ideals in
  \[
    \cP_X = \{ (i,j) \in \Z^2 \mid 1 \leq i \leq j \leq N-1 \} \,,
  \]
  and elements of $\cB$ correspond to order ideals in
  \[
    \wh\cP_X = \{ (i,j) \in \Z^2 \mid i \leq j \leq i+N-2 \} \,,
  \]
  see \cite{buch.chaput.ea:positivity} and \Figure{ogcyl}. Multiplication by
  $[X^{N-1}]$ is given by translating borders diagonally by one unit, and
  multiplication by a point is given by the same procedure as for Lagrangian
  Grassmannians, except the reflected shape will be attached to $\cP_X$ one row
  lower. The number $\coeff(\mE, [X^{N-1}]^{\star N/2})$ is equal to the number
  of order ideals $\la \subset \cP_X$ for which $[X^{N-1}]^{\star(-N/2)} \star
  \mP \star [X^\la] = [X^\la]$. This holds if and only if the border of $\la$ is
  a path from the upper-right corner of $\cP_X$ to the middle outer corner of
  the south-west side of $\cP_X$ that is symmetric under reflection in the
  diagonal. There are $2^{N/2}$ such paths.
\end{proof}

\begin{figure}[h]
  \caption{The partially ordered set $\wh\cP_X$ for $X = \OG(6,12)$, with the
    boxes of $\cP_X$ colored gray.}
  \label{fig:ogcyl}
  \begin{center}
    \begin{tikzpicture}[x=4mm,y=4mm]
      \def\zz{-- ++(0,-1) -- ++(1,0)}
      \draw [fill=gray!30] (5,0) -- (0,0) \zz\zz\zz\zz\zz -- cycle;
      \foreach \i in {1,...,4} \draw (\i,0) -- (\i,-\i) -- (5,-\i);
      \draw (-2,2) \zz\zz\zz\zz\zz\zz\zz\zz -- ++(0,-1);
      \draw (3,2) \zz\zz\zz\zz\zz\zz\zz\zz -- ++(0,-1);
    \end{tikzpicture}
  \end{center}
\end{figure}
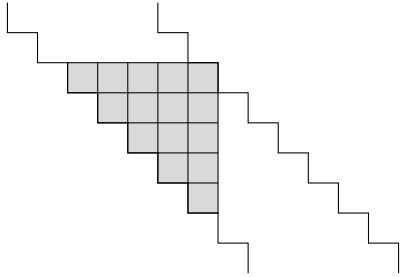

\begin{cor}
  Let $n,d \geq 0$ satisfy $4d = nN$. Then,
  \[
    \vTev^{\OG(N,2N)}_{1,d,n} = \begin{cases}
      \coeff(\mE,\mP) = 2^{N-1} & \text{if $n$ is even;} \\
      \coeff(\mE,q^{N/4}) = 2^{N/2} & \text{if $n$ is odd.}
    \end{cases}
  \]
\end{cor}

\subsection{Exceptional cominuscule flag varieties}

There are two exceptional cominuscule flag varieties, the Cayley plane $E_6/P_6$
and the Freudenthal variety $E_7/P_7$. The quantum cohomology of these spaces is
known from \cite{chaput.manivel.ea:quantum*1}. The following results account for
the virtual Tevelev degrees of genus 1, and use the standard ordering of the
simple roots:
\[
  \pic{.9}{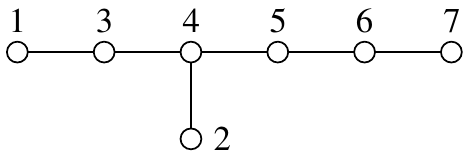}
\]
The quantum Euler classes were obtained using the Equivariant Schubert
Calculator \cite{buch:equivariant}. More general formulas for the virtual
Tevelev degrees of arbitrary genus are included among the examples in
\Section{furthercfvs}.

\begin{thm}
  Let $X = E_6/P_6$ be the Cayley plane. We have $\dim(X) = 16$ and $\mP^{\star
  3} = q^4$ in $\QH(X)$. For $n,d \geq 0$ satisfying $3d=4n$, we have
  \[
    \vTev^X_{1,d,n} = \coeff(\mE, \mP) = \chi(E_6/P_6) = 27 \,.
  \]
  The quantum Euler class of the Cayley plane is
  \[
    \mE = 27\, \mP +
    27\, q [X^{s_2 s_4 s_5 s_6}] + 45\, q [X^{s_3 s_4 s_5 s_6}] \,.
  \]
\end{thm}

\begin{thm}
  Let $X = E_7/P_7$ be the Freudenthal variety. We have $\dim(X) = 27$ and
  $\mP^{\star 2} = q^3$ in $\QH(X)$. For $n,d \geq 0$ satisfying $2d=3n$, we
  have
  \[
    \vTev^X_{1,d,n} = \coeff(\mE, \mP) = \chi(E_7/P_7) = 56 \,.
  \]
  The quantum Euler class of the Freudenthal variety is
  \[
    \mE = 56\, \mP + 160\, q [X^{s_6 s_5 u}] + 272\, q [X^{s_5 s_1 u}]
    + 160\, q [X^{s_3 s_1 u}] \,,
  \]
  where $u = s_4 s_3 s_2 s_4 s_5 s_6 s_7$.
\end{thm}


\section{Strange symmetry of $\mE/\mP$}\label{sec:strange}

\subsection{Overview}

Let $X = G/P_X$ be a cominuscule flag variety. Quantum multiplication by
$\mE/\mP$ preserves the degree grading of $\QH(X)_q$. Let
\[
  [\mE/\mP]_k : \QH(X)_{q,k} \rightarrow \QH(X)_{q,k}
\]
denote the restriction of quantum multiplication by $\mE/\mP$ to the degree $k$
subspace $\QH(X)_{q,k}\subset \QH(X)_{q}$. A {\em strange inner product} on
$\QH(X)_q$ which respects the degree grading is defined by the {\em strange
duality involution}
\[
  \iota : \QH(X)_q \to \QH(X)_q
\]
of \cite{postnikov:affine, chaput.manivel.ea:quantum}. We will prove that
quantum multiplication by $\mE/\mP$ in $\QH(X)_q$ is self-adjoint with respect
to the strange inner product and that {\em all} virtual Tevelev degrees of $X$
can be expressed in terms of the eigenvalues and eigenvectors of the degree 0
operator $[\mE/\mP]_0$.

\subsection{Strange duality}

Let $\al_0 \in \Phi$ denote the highest root. Given any simple root $\be \in
\Phi^S$, let $n_\be(\al_0)$ denote the coefficient of $\be$ obtained when
$\al_0$ is expanded in the basis of simple roots, and define
\[
  \epsilon(\be) = \begin{cases}
    1 & \text{if $\be$ is a long root,} \\
    -1 & \text{if $\be$ is a short root.}
  \end{cases}
\]
If the root system $\Phi$ is simply-laced, then all roots are long by
convention. Given a minimal representative $u \in W^X$ and a reduced expression
$u = s_{\be_1} s_{\be_2} \cdots s_{\be_\ell}$ we set
\[
  y(u) = \prod_{i=1}^\ell n_{\be_i}(\al_0)^{\epsilon(\be_i)} \,.
\]
Minimal representatives of cominuscule Schubert varieties are known to be fully
commutative \cite{stembridge:fully}, in the sense that any reduced expression
can be obtained from any other by interchanging commuting simple reflections. It
follows that the rational number $y(u)$ does not depend on the chosen reduced
expression of $u$. More generally, given any Weyl group element $w \in W$, we
set
\[
  y(w) = y(u)
\]
where $u \in W^X \cap w W_X$ is the unique minimal representative of the coset
$w W_X$. If $X = \Gr(m,N)$ is a Grassmannian of type A, then $y(w) = 1$ for all
$w \in W$.

Let $\delta(u)$ denote the minimal degree of a rational curve in $X$ through
$1.P_X$ and $u.P_X$. When $u \in W^X$, this is the number of occurrences of the
cominuscule simple root $\ga$ in any reduced expression of $u$
\cite[Prop.~18]{chaput.manivel.ea:quantum*1}, and by
\cite[Thm.~1]{chaput.manivel.ea:affine} we have
\begin{equation}\label{eqn:multP}
  \mP \star [X^u] = q^{\delta(u)} [X^{w_0^X u}] \,.
\end{equation}
Let $w_{0,X} \in W_X$ be the longest element in the Weyl group of $P_X$, and let
$w_0^X \in W^X$ be the minimal representative of the point class. Then $w_0^X
w_{0,X} = w_0$ is the longest element in $W$.

The involution $\iota$ of the following result is called the \emph{strange
duality involution}. It was constructed for Grassmannians of type A in
\cite[Thm.~6.5]{postnikov:affine} and generalized to all cominuscule flag
varieties in \cite[Thm.~1.1]{chaput.manivel.ea:quantum}.

\begin{thm}\label{thm:strange}
  Let $X$ be any cominuscule flag variety. There is a well-defined ring
  involution $\iota : \QH(X)_q \to \QH(X)_q$ given by
  \[
    \iota(q) = y(s_{\al_0})\, q^{-1}
    \ \ \ \ \ \text{and} \ \ \ \ \
    \iota [X^u] = y(u)\, q^{-\delta(u)}\, [X^{w_{0,X}u}] \,.
  \]
\end{thm}

Given $u \in W^X$ we denote the dual Weyl group element by $u^\vee = w_0 u
w_{0,X} \in W^X$. We have $X^{u^\vee} = w_0.X_u$ as subvarieties of $X$, in
particular $[X^{u^\vee}] = [X_u]$.

\begin{lemma}\label{lemma:strange}%
  For $u \in W^X$ we have $\iota [X^u] = y(u)\, [X_u]/\mP$.
\end{lemma}
\begin{proof}
  By \Theorem{strange} and \eqn{multP} we have
  \[
    \mP \star \iota [X^u] = y(u)\, q^{-\delta(u)}\, \mP \star [X^{w_{0,X} u}]
    = y(u)\, [X^{w_0^X w_{0,X} u}] = y(u)\, [X_u] \,,
  \]
  which is equivalent to the Lemma.
\end{proof}

\begin{defn}\label{defn:innerprod}%
  For $\mA, \mB \in \QH(X)_q$, the {\em strange
  inner product} is
  \[
    (\mA, \mB) = \coeff(\mA \star \iota(\mB), 1) \,.
  \]
  The strange norm of $\mA$ is the real number $|\mA| = \sqrt{(\mA, \mA)}$.
\end{defn}

The strange inner product is the least strange in type A, where the Schubert
basis $\cB$ is an orthonormal basis of $\QH(X)_q$.

\begin{cor}\label{cor:innerprod}%
  The pairing $(-,-)$ of \Definition{innerprod} is an inner product on the
  $\Q$-vector space $\QH(X)_q$. The set $\cB = \{ q^d [X^u] \mid u \in W^X, d
  \in \Z \}$ is an orthogonal basis of $\QH(X)_q$, and the basis elements have
  norms given by
  \[
    |\,q^d [X^u]\,|^2 = y(s_{\al_0})^d\, y(u) \,.
  \]
\end{cor}
\begin{proof}
  Since $\iota(1) = 1$ and $\iota(q^d [X^u])$ is a non-zero multiple of an
  element of $\cB$ for all $q^d [X^u] \in \cB$, it follows that
  \[
    \coeff(\iota(A), 1) = \coeff(A, 1)
  \]
  for each $\mA \in \QH(X)_q$. This implies that the pairing $(-,-)$ is
  symmetric, and the pairing is bilinear by definition. For $u, v \in W^X$ and
  $d,e \in \Z$ we have
  \[
    \begin{split}
      (q^d [X^u], q^e [X^v])
      &= \coeff(q^d [X^u] \star \iota(q^e [X^v]) \star \mP, \mP) \\
      &= y(s_{\al_0})^e\, y(v)\, \coeff([X^u] \star [X_v], q^{e-d} \mP) \\
      &= y(s_{\al_0})^e\, y(v)\, \langle [X^u], [X_v], 1 \rangle^X_{0,e-d} \\
      &= y(s_{\al_0})^e\, y(v)\, \delta_{u,v}\, \delta_{d,e} \,.
    \end{split}
  \]
  This shows that $\cB$ is an orthogonal basis of $\QH(X)_q$ and reveals the
  norms of the basis elements.
\end{proof}

\begin{lemma}\label{lemma:symmetry}
  For $u \in W^X$ we have $y(u^\vee)\, y(u) = y(w_0^X)$.
\end{lemma}
\begin{proof}
  Let $u = s_{\be_1} s_{\be_2} \cdots s_{\be_\ell}$ and $u^\vee = s_{\al_1}
  s_{\al_2} \cdots s_{\al_k}$ be reduced expressions. We have
  \[
    (w_0 u^\vee w_0)^{-1} u =
    w_0 w_{0,X} u^{-1} w_0 w_0 u = w_0 w_{0,X} = w_0^X \,.
  \]
  Since $\ell(w_0^X) = \ell(u^\vee) + \ell(u)$, we deduce that
  \[
    w_0^X = s_{-w_0.\al_k} \cdots s_{-w_0.\al_1} s_{\be_1} \cdots s_{\be_\ell}
  \]
  is a reduced expression for $w_0^X$. The Lemma follows from this by observing
  that $n_\be(\al_0) = n_{-w_0.\be}(-w_0.\al_0) = n_{-w_0.\be}(\al_0)$ for any
  simple root $\be$.
\end{proof}

\begin{prop}\label{prop:symmetry}%
  We have $\iota(\mE/\mP) = \mE/\mP$.
\end{prop}
\begin{proof}
  For $u \in W^X$ we have by \Lemma{strange} and \Lemma{symmetry} that
  \[
    \iota([X_u] \star [X^u]/\mP)
    = y(u^\vee) [X^u]/\mP \star y(u) [X_u]/\mP \star y(w_0^X)^{-1} \mP
    = [X^u] \star [X_u]/\mP \,.
  \]
  The Proposition follows from this by taking the sum over $u \in W^X$.
\end{proof}

\subsection{Virtual Tevelev degrees}

Given any element $\mA \in \QH(X)_q$, we will abuse notation and identify $\mA$
with the linear endomorphism
\[
  \mA \star : \QH(X)_q \to \QH(X)_q
\]
defined by quantum multiplication by $\mA$. Statements about diagonalizability,
eigenvalues, and eigenvectors of $\mA$ should be interpreted in this sense. The
real numbers $\R$ can be replaced by an appropriate finite extension of $\Q$ in
the following result.

\begin{cor}\label{cor:diagonalize}%
  The vector space $\QH(X)_q \otimes \R$, endowed with the strange inner product
  $(-,-)$, has an orthogonal basis consisting of eigenvectors of $\mE/\mP$.
\end{cor}
\begin{proof}
  For $\mA, \mB \in \QH(X)_q$, \Proposition{symmetry} shows that
  \[
    (\mE/\mP \star \mA\,,\, \mB) = (\mA\,,\, \mE/\mP \star \mB) \,,
  \]
  that is, the endomorphism $\mE/\mP$ is self-adjoint with respect to the inner
  product on $\QH(X)_q$. Since $\mE/\mP$ has degree zero in the graded ring
  $\QH(X)_q$, it preserves the homogeneous components $\QH(X)_{q,k}$ of this
  ring. Since these components have finite dimension, it follows that $\mE/\mP$
  is diagonalizable over $\R$ by an orthogonal basis.
\end{proof}

\begin{thm}\label{thm:vTev_comin}%
  Let $X$ be a cominuscule flag variety, and let
  \[
    \mA_1, \mA_2, \dots, \mA_k
  \]
  be any orthogonal basis of $\QH(X)_{q,0} \otimes \R$ with respect to the
  strange inner product, consisting of eigenvectors of $[\mE/\mP]_0$, and let
  \[
    \la_1, \la_2, \dots, \la_k \in \R
  \]
  be the corresponding eigenvalues. For all $g,d,n \geq 0$, we have
  \[
    \begin{split}
      \vTev^X_{g,d,n} \
      &= \ \coeff((\mE/\mP)^{\star g}, q^d \mP^{\star (1-n-g)}) \\
      &= \ \sum_{i=1}^k
      \frac{\coeff(\mA_i, 1) \coeff(\mA_i, q^d \mP^{\star(1-n-g)})}{|A_i|^2}\,
      \la_i^g
      \,.
    \end{split}
  \]
\end{thm}
\begin{proof}
  The vector space $\QH(X)_{q,0} \otimes \R$ has an orthogonal basis of
  eigenvectors of $[\mE/\mP]_0$ by \Corollary{diagonalize}. The formula follows
  from the identity
  \[
    (\mE/\mP)^{\star g} \star 1
    = (\mE/\mP)^{\star g} \star \sum_{i=1}^k \frac{(\mA_i, 1)}{|A_i|^2}\, \mA_i
    = \sum_{i=1}^k \frac{\coeff(\mA_i, 1)}{|A_i|^2}\, \la_i^g\, \mA_i
  \]
  by extracting the coefficient of $q^d \mP^{\star(1-n-g)}$ on both sides.
\end{proof}

It was proved in \cite{chaput.manivel.ea:quantum*2} that the ring
$\QH(X)/\langle q-1 \rangle$ is semi-simple, which implies that quantum
multiplication by any element is diagonalizable over $\C$. In particular, part
(b) of the following result was known.

\begin{cor}\label{cor:diagonalizeE}
  Let $X$ be a cominuscule flag variety.\smallskip

  \noindent{\rm(a)} Quantum multiplication by $\mE/\mP$ on $\QH(X)/\langle q-1
  \rangle$ is diagonalizable over $\R$.\smallskip

  \noindent{\rm(b)} Quantum multiplication by $\mE$ on $\QH(X)/\langle q-1
  \rangle$ is diagonalizable over $\C$.\smallskip

  \noindent{\rm(c)} Quantum multiplication by $\mE^{\ord(\mP)}$ on
  $\QH(X)/\langle q-1 \rangle$ is diagonalizable over $\R$.
\end{cor}
\begin{proof}
  Part (a) follows from \Corollary{diagonalize}. Since $\mP$ is a Seidel class,
  quantum multiplication by $\mP$ is an idempotent operation on $\QH(X)/\langle
  q-1 \rangle$, hence diagonalizable over $\C$. Part (b) therefore follows from
  (a) and the commutativity of the ring $\QH(X)$. Part (c) follows from (a)
  since $\mE^{\star \ord(P)} = (\mE/\mP)^{\star \ord(P)}$ in $\QH(X)/\langle q-1
  \rangle$.
\end{proof}

\Theorem{vTev_comin} is more efficient in practice to use than semi-simplicity
since only the eigenvalues and eigenvectors of a single operator $[\mE/\mP]_0$
on $\QH(X)_{q,0}$ are required instead of the entire system of idempotents of
$\QH(X)$.

\begin{remark}\label{remark:isone}%
  The cominuscule flag variety $X$ may satisfy the property that
  \begin{equation}\label{eqn:isone}
    \int_d c_1(T_X) = r (n+g-1)
    \ \ \ \ \ \text{implies} \ \ \ \ \
    q^d \mP^{\star(1-n-g)} = 1
  \end{equation}
  for all $g,d,n \in \Z$, where $r = \dim(X)$. If \eqn{isone} holds, then the
  virtual Tevelev degree $\vTev^X_{g,d,n}$ depends only on $g$ when the
  dimension constraint \eqn{bbb} is satisfied. Property \eqn{isone} holds for
  $X$ if and only if
  \[
    \ord(P) \ \ \text{divides} \ \ \frac{d_q}{\gcd(d_q,r)} \,,
    \ \ \ \ \ \ \ \text{where} \ \
    d_q = \frac{1}{2} \deg(q) = \int_\rmline c_1(T_X) \,,
  \]
  with the integral over the positive generator of $H_2(X,\Z)$. When $X =
  \Gr(m,N)$ is a Grassmannian, the above condition is equivalent to
  \[
    \gcd(r,N) \ \ \text{divides} \ \ m \,,
  \]
  where $r = \dim(X) = m(N-m)$. The Grassmannians of dimension $r<20$ that do
  not satisfy this condition are $\Gr(2,4)$, $\Gr(2,8)$, $\Gr(4,8)$, and
  $\Gr(3,9)$. Property \eqn{isone} fails for all quadrics $Q^r$, is satisfied
  for the Lagrangian Grassmannian $\LG(N,2N)$ if and only if $N$ is odd, and is
  satisfied for the maximal orthogonal Grassmannian $\OG(N,2N)$ if and only if
  $N$ is not divisible by 4.
\end{remark}

\subsection{Grassmannians}\label{sec:examex}

\Theorem{vTev_comin} provides an effective method to calculate virtual Tevelev
degrees. After revisiting the quadric $Q^r$, including $Q^4=\Gr(2,4)$, we fully
calculate $\Gr(2,5)$, $\Gr(2,6)$, and $\Gr(3,6)$. For $\Gr(2,8)$, $\Gr(3,8)$,
and $\Gr(4,8)$, we just present the final formulas (obtained by computer algebra
and using the Equivariant Schubert Calculator \cite{buch:equivariant}).

\begin{example}
  Let $X = Q^r$ be the quadric of dimension $r$. By \Section{quadric} we have
  \[
    \mE/\mP = (r+\delta) + (r-\delta)\, q^{-1} \mP \,,
  \]
  where $\delta = 1$ if $r$ is odd, whereas $\delta=2$ if $r$ is even. An
  orthogonal basis of $\QH(X)_{q,0}$ consisting of eigenvectors of $\mE/\mP$ is
  given by
  \[
    \mA_1 = 1 + q^{-1} \mP
    \ \ \ \ \ \ \text{and} \ \ \ \ \ \
    \mA_2 = 1 - q^{-1} \mP \,,
  \]
  and the corresponding eigenvalues are
  \[
    \la_1 = 2r \ \ \ \ \ \ \text{and} \ \ \ \ \ \ \la_2 = 2\delta \,.
  \]
  Noting that
  \[
    \coeff(\mA_1, q^{d} \mP^{\star (-d)}) = 1
    \ \ \ \ \ \ \text{and} \ \ \ \ \ \
    \coeff(\mA_2, q^{d} \mP^{\star(-d)}) = (-1)^d\,,
  \]
  we obtain from \Theorem{vTev_comin} that
  \[
    \vTev^{Q^r}_{g,d,n} \ = \
    \frac{1}{2} {\la_1}^g + \frac{(-1)^d}{2} {\la_2}^g \ = \
    \frac{(2r)^g + (-1)^d (2\delta)^g}{2}
  \]
  whenever $g,d,n \geq 0$ satisfy $d = n+g-1$.
\end{example}

\begin{example}
  Let $X = \Gr(2,5)$. Quantum multiplication by $\mE/\mP$ on the basis
  \[
    \{ 1, \qXcb \}
  \]
  of $\QH(X)_{q,0}$ is given by
  \begin{align*}
    \mE/\mP \star 1 &= 10 + 5\, \qXcb \,, \\
    \mE/\mP \star \qXcb &= 5 + 15\, \qXcb \,.
  \end{align*}
  The eigenvalues of $[\mE/\mP]_0$ are
  \[
    \la_1 = \frac{25 + 5\,\sqrt{5}}{2}
    \ \ \ \ \ \text{and} \ \ \ \ \
    \la_2 = \frac{25 - 5\,\sqrt{5}}{2} \,,
  \]
  and corresponding eigenvectors are
  \[
    \mA_1 = 2 + \left( 1 + \sqrt{5} \right)\, \qXcb
    \ \ \ \ \ \text{and} \ \ \ \ \
    \mA_2 = 2 + \left( 1 - \sqrt{5} \right)\, \qXcb \,.
  \]
  For $g,d,n \geq 0$ satisfying $5d = 6(n+g-1)$, we obtain from
  \Theorem{vTev_comin} and \Remark{isone} that
  \[
    \begin{split}
      \vTev^{\Gr(2,5)}_{g,d,n} \
      &= \ \sum_{i=1}^2 \frac{\coeff(\mA_i,1)^2}{|\mA_i|^2}\, \la_i^g \\
      &= \ \frac{5 - \sqrt{5}}{10} \left( \frac{25 + 5\,\sqrt{5}}{2} \right)^g
      + \frac{5 + \sqrt{5}}{10} \left( \frac{25 - 5\,\sqrt{5}}{2} \right)^g \,.
    \end{split}
  \]
\end{example}

\begin{example}
  Let $X = \Gr(2,6)$. Quantum multiplication by $\mE/\mP$ on the Schubert basis
  of $\QH(X)_{q,0}$ is given by
  \begin{align*}
    \mE/\mP \star 1 &= 15 + 3\, \qXcc + 9\, \qXdb \,, \\
    \mE/\mP \star \qXcc &= 3 + 15\, \qXcc + 9\, \qXdb \,, \\
    \mE/\mP \star \qXdb &= 9 + 9\, \qXcc + 27\, \qXdb \,.
  \end{align*}
  The eigenvalues and associated eigenvectors of $[\mE/\mP]_0$ are:
  \begin{align*}
    \la_1 &= 36 & \mA_1 &= 1 + \qXcc + 2\, \qXdb \\
    \la_2 &= 12 & \mA_2 &= 1 - \qXcc \\
    \la_3 &= 9 & \mA_3 &= 1 + \qXcc - \qXdb
  \end{align*}
  For $g,d,n \geq 0$ satisfying $6d = 8(n+g-1)$, we obtain from
  \Theorem{vTev_comin} and \Remark{isone} that
  \[
    \vTev^{\Gr(2,6)}_{g,d,n}
    \ = \ \sum_{i=1}^3 \frac{\coeff(\mA_i,1)^2}{|\mA_i|^2}\, \la_i^g
    \ = \ \frac{{36}^g}{6} + \frac{{12}^g}{2} + \frac{9^g}{3} \,.
  \]
\end{example}

\begin{example}
  Let $X = \Gr(3,6)$. Quantum multiplication by $\mE/\mP$ on the Schubert basis
  of $\QH(X)_{q,0}$ is given by
  \begin{align*}
    \mE/\mP \star 1 &= 20 + 2\, \qXcc + 2\, \qXbbb + 16\, \qXcba \,, \\
    \mE/\mP \star \qXcc &= 2 + 20\, \qXcc + 2\, \qXbbb + 16\, \qXcba \,, \\
    \mE/\mP \star \qXbbb &= 2 + 2\, \qXcc + 20\, \qXbbb + 16\, \qXcba \,, \\
    \mE/\mP \star \qXcba &= 16 + 16\, \qXcc + 16\, \qXbbb + 56\, \qXcba \,.
  \end{align*}
  The eigenvalues and associated orthogonal eigenvectors of $[\mE/\mP]_0$ are:
  \begin{align*}
    \la_1 &= 72 &
    \mA_1 &= 1 + \qXcc + \qXbbb + 3\, \qXcba \\
    \la_2 &= 18 &
    \mA_2 &= 2 - \qXcc - \qXbbb \\
    \la_3 &= 18 &
    \mA_3 &= \qXcc - \qXbbb \\
    \la_4 &= 8 &
    \mA_4 &= 1 + \qXcc + \qXbbb - \qXcba \,.
  \end{align*}
  For $g,d,n \geq 0$ satisfying $6d = 9(n+g-1)$, we obtain from
  \Theorem{vTev_comin} and \Remark{isone} that
  \[
    \vTev^{\Gr(3,6)}_{g,d,n}
    \ = \ \sum_{i=1}^4 \frac{\coeff(\mA_i,1)^2}{|\mA_i|^2}\, \la_i^g
    \ = \ \frac{{72}^g}{12} + \frac{2 \cdot {18}^g}{3} + \frac{8^g}{4} \,.
  \]
\end{example}
\vspace{8pt}

\begin{example}
The virtual Tevelev degrees of $\Gr(2,8)$, $\Gr(3,8)$, and $\Gr(4,8)$ are given
by the following formulas (valid when the dimension constraint \eqn{bbb} is
satisfied):

\vspace{12pt}
\noindent $\bullet$
For
$\Gr(2,8)$, the virtual Tevelev degrees are
\[
  \frac{\left( 2-\sqrt{2} \right) \left( 64 + 32\,\sqrt{2} \right)^g}{8}
  + \frac{\left( 2+\sqrt{2} \right) \left( 64 - 32\,\sqrt{2} \right)^g}{8}
  + \frac{(-1)^d {32}^{g}}{4}
  + \frac{(-1)^d {16}^{g}}{4}
  \,.
\]

\vspace{8pt}
\noindent $\bullet$
For $\Gr(3,8)$, the virtual Tevelev degrees are
\[
  \frac{\left(3 - 2\,\sqrt{2} \right) \left(384 + 256\,\sqrt{2} \right)^g}{16}
  + \frac{\left( 3+2\,\sqrt{2} \right) \left(384 - 256\,\sqrt{2} \right)^g}{16}
  + \frac{{128}^g}{8}
  + \frac{{64}^g}{4}
  + \frac{{32}^g}{4}
  \,.
\]

\vspace{8pt}
\noindent $\bullet$
For $\Gr(4,8)$, the virtual Tevelev degrees are
\begin{multline*}
  \frac{\left( 3-2\,\sqrt{2} \right) \left( 768+512\,\sqrt{2} \right)^g}{32}
  + \frac{\left( 3+2\,\sqrt{2} \right) \left( 768-512\,\sqrt{2} \right)^g}{32}
  + \frac{{128}^{g}}{8}
  + \frac{{64}^{g}}{16}
  + \frac{{16}^{g}}{8} \\
  + \frac{(-1)^{d/2}\left(2-\sqrt{2}\right) \left(128+64\,\sqrt{2}\right)^g}{8}
  + \frac{(-1)^{d/2}\left(2+\sqrt{2}\right) \left(128-64\,\sqrt{2}\right)^g}{8}
  \,.
\end{multline*}
For $\Gr(4,8)$, the dimension constraint \eqn{bbb} implies that $d$ is even.
\end{example}

Formulas for $\Gr(2,7)$ and $\Gr(3,7)$ have complexity similar to that of the
Freudenthal variety $E_7/P_7$ below, so they are omitted. Some virtual Tevelev
degrees of these spaces are included in the following table.\medskip

\begin{center}
\begin{tabular}{|c|c|c|}
  \hline
  $g$ & $\Gr(2,7)$ & $\Gr(3,7)$ \\
  \hline
  0 & 1 & 1 \\
  1 & 21 & 35 \\
  2 & 686 & 2744 \\
  3 & 33614 & 470596 \\
  4 & 2000033 & 107884133 \\
  5 & 126825622 & 26310551764 \\
  6 & 8191782221 & 6491563697269 \\
  7 & 531900893867 & 1605160235412769 \\
  8 & 34589376715299 & 397071802007102691 \\
  9 & 2250344155712982 & 98232421880349925476 \\
  10 & 146424292089662006 & 24302307748473316398284 \\
  11 & 9527847961374037099 & 6012312236720159623681561 \\
  12 & 619985909132445247770 & 1487427484539611374221472752 \\
  13 & 40343209216871520541603 & 367985011574983125611827761985 \\
  14 & 2625182876113221414704217 & 91038368842060169714846533326833 \\
  15 & 170823979704176185099894853 & 22522614725296806700134311109583811 \\
  \hline
\end{tabular}
\end{center}
\medskip


\subsection{Further cominuscule flag varieties}\label{sec:furthercfvs}%

The first Lagrangian Grassmannians do not produce new examples, $\LG(1,2) \cong
\bP^1$ and $\LG(2,4) \cong Q^3$. The following formulas are valid when the
dimension constraint \eqn{bbb} is satisfied:

\vspace{12pt}
\noindent $\bullet$
For
$\LG(3,6)$, the virtual Tevelev degrees are
\[
  \frac{\left( 2 - \sqrt{2} \right) \left( 16+8\,\sqrt{2} \right)^g}{4}
  + \frac{\left( 2 + \sqrt{2} \right) \left( 16-8\,\sqrt{2} \right)^g}{4}
  \,.
\]

\vspace{8pt}
\noindent $\bullet$
For $\LG(4,8)$, the virtual Tevelev degrees are
\[
  \frac{\left( 5-2\,\sqrt{5} \right) \left( 100+40\,\sqrt{5} \right)^g}{20}
  + \frac{\left( 5+2\,\sqrt{5} \right) \left( 100-40\,\sqrt{5} \right)^g}{20}
  + \frac{(-1)^{d/2}\,{20}^g}{4}
  + \frac{(-1)^{d/2}\,{4}^g}{4}
  \,.
\]

\vspace{8pt}
\noindent $\bullet$
For
$\LG(5,10)$, the virtual Tevelev degrees are
\[
  \frac{\left( 7-4\,\sqrt{3} \right) \left( 1008+576\,\sqrt{3} \right)^g}{24}
  + \frac{\left( 7+4\,\sqrt{3} \right) \left( 1008-576\,\sqrt{3} \right)^g}{24}
  + \frac{{144}^g}{24}
  + \frac{{48}^g}{4}
  + \frac{{16}^g}{8}
  \,.
\]

\vspace{8pt}

The first orthogonal Grassmannians do not produce new examples, $\OG(2,4) \cong
\bP^1$ and $\OG(3,6) \cong \bP^3$. The following formulas are valid when the
dimension constraint \eqn{bbb} is satisfied:

\vspace{12pt}
\noindent $\bullet$
For $\OG(4,8)$, the virtual Tevelev degrees are
\[
  \frac{{12}^g}{2} + \frac{(-1)^d\, 4^g}{2} \,.
\]

\vspace{8pt}
\noindent $\bullet$
For
$\OG(5,10)$, the virtual Tevelev degrees are
\[
  \frac{ \left( 2-\sqrt{2} \right) \left( 32+16\,\sqrt{2} \right)^g}{4}
  + \frac{ \left( 2+\sqrt{2} \right) \left( 32-16\,\sqrt{2} \right)^g}{4}
  \,.
\]
\vspace{8pt}

\noindent $\bullet$
For $\OG(6,12)$, the virtual Tevelev degrees are
\[
  \frac{\left( 5-2\,\sqrt{5} \right) \left( 200+80\,\sqrt{5} \right)^g}{20}
  + \frac{\left( 5+2\,\sqrt{5} \right) \left( 200-80\,\sqrt{5} \right)^g}{20}
  + \frac{{40}^g}{4}
  + \frac{{8}^g}{4}
  \,.
\]
\vspace{8pt}

Finally, there are two exceptional cases: the Cayley plane $E_6/P_6$ and the
Freudenthal variety $E_7/P_7$. The following formulas are valid when the
dimension constraint \eqn{bbb} is satisfied:

\vspace{12pt}
\noindent $\bullet$
For $E_6/P_6$, the virtual Tevelev degrees are
\[
  \frac{\left( 2-\sqrt{3} \right) \left( 144+72\,\sqrt{3} \right)^g}{6}
  + \frac{\left( 2+\sqrt{3} \right) \left( 144-72\,\sqrt{3} \right)^g}{6}
  + \frac{9^g}{3}
  \,.
\]

\vspace{12pt}
\noindent $\bullet$
For $E_7/P_7$, the virtual Tevelev degrees are
\begin{multline*}
  \frac{8^g}{4}
  + \frac{a_1+a_2\zeta+ a_3 \zeta^{-1}}{b_1 + b_2 \zeta + b_3 \zeta^{-1}}
  \left( 2376 + 432\,\zeta + 12096\,\zeta^{-1} \right)^g
  \\
  + \frac{a_1 - a_4 \zeta + a_5 \zeta^{-1}}{b_1 - b_4 \zeta + b_5 \zeta^{-1}}
  \big( 2376 - \alpha + \beta \big)^g
  + \frac{a_1 + a_6 \zeta - a_7 \zeta^{-1}}{b_1 + b_6 \zeta - b_7 \zeta^{-1}}
  \big( 2376 - \alpha - \beta \big)^g
\end{multline*}
where
\[
\zeta=\sqrt[3]{148+4\,i\sqrt{3}}
\]
is taken in the first quadrant, and the other constants are defined by:
\begin{align*}
\alpha &= 864\,\sqrt{7}\, \cos\left(\frac{1}{3}\,
\arctan \left( \frac{\sqrt{3}}{37} \right) \right) &
\beta &= 864\,\sqrt{21}\, \sin \left(\frac{1}{3}\,
\arctan \left( \frac{\sqrt{3}}{37} \right) \right) \\
a_{{1}}&=156498345\,i\sqrt{3}+1918858850 &
b_{{1}}&=60400326564\,i\sqrt{3}+740581002120 \\
a_{{2}}&=26312126\,i\sqrt{3}+363365382 &
b_{{2}}&=10153341360\,i\sqrt{3}+140247740712 \\
a_{{3}}&=919199848\,i\sqrt{3}+10129587384 &
b_{{3}}&=354845016000\,i\sqrt{3}+3909674626464 \\
a_{{4}}&=194838754\,i\sqrt{3}+142214502 &
b_{{4}}&=75200541036\,i\sqrt{3}+54893858316 \\
a_{{5}}&=4605193768\,i\sqrt{3}-6443593464 &
b_{{5}}&=1777414805232\,i\sqrt{3}-2487104837232 \\
a_{{6}}&=168526628\,i\sqrt{3}-221150880 &
b_{{6}}&=65047199676\,i\sqrt{3}-85353882396 \\
a_{{7}}&=5524393616\,i\sqrt{3}+3685993920 &
b_{{7}}&=2132259821232\,i\sqrt{3}+1422569789232
\end{align*}

\Theorem{vTev_comin} implies that all eigenvalues of $\mE/\mP$ and all
coefficients in the formula for virtual Tevelev degrees are real numbers (which is
not clear from the above expressions). The following table contains some
virtual Tevelev degrees of $E_7/P_7$.\smallskip

\begin{center}
  \begin{tabular}{|c|c|c|}
    \hline
    $g$ & $E_7/P_7$ \\
    \hline
    0 & 1 \\
    1 & 56 \\
    2 & 128320 \\
    3 & 869201408 \\
    4 & 6035673223168 \\
    5 & 41931214470742016 \\
    6 & 291308765400165253120 \\
    7 & 2023810102768684733825024 \\
    8 & 14060020975152452459315593216 \\
    9 & 97679218802247250296546711830528 \\
    10 & 678607080508699448610546779756167168 \\
    11 & 4714488663641616811439032212948871282688 \\
    12 & 32752978856253489427845306031022643827703808 \\
    13 & 227544851731504006840105249380606108740637163520 \\
    14 & 1580822916191834644483662867101537620104827373617152 \\
    15 & 10982454990043024221511165369579640911620064101974147072 \\
    \hline
  \end{tabular}
\end{center}
\medskip


\section{Complete intersections}\label{sec:cplint}

\subsection{Classical cohomology}

Let $X = V(f_1,\dots,f_L) \subset \bP^{r+L}$ be a nonsingular complete
intersection of dimension $r$ where
\[
  f_i \in \Gamma(\bP^{r+L},\cO(m_i))
\]
for $1 \leq i \leq L$. We denote the vector of degrees by $\mm =
(m_1,\dots,m_L)$. We will always assume $r\geq 3$, so
\[
  H_2(X,\Z) = \Z
\]
by the Lefschetz Hyperplane Theorem.

The following notation will be convenient, for any $a,b \in \Z$:
\[
  |\mm| = \sum_{i=1}^L m_i \,, \ \ \ \
  \mm^{a\mm+b} = \prod_{i=1}^L m_i^{a m_i + b} \,, \ \ \ \
  (a\mm+b)! = \prod_{i=1}^L (a m_i + b)! \ .
\]
For example, the projective degree of $X$ in $\bP^{r+L}$ is $\mm^1 =
\prod_{i=1}^L m_i$, and $X$ is Fano if and only if $|\mm| \leq r+L$.

Let $H^*(X) = H^*(X,\Q)$ be the singular cohomology ring with rational
coefficients. The restricted part of $H^*(X)$ has basis $\{1, \mH, \dots, \mH^r
\}$, where $\mH\in H^*(X)$ denotes the image of the hyperplane class $\mH \in
H^2(\bP^{r+L})$ under the restriction
\[
  H^2(\bP^{r+L}) \rightarrow H^2(X) \,.
\]
The class of a point is
\[
  \mP = \mm^{-1} \mH^r \in H^{2r}(X) \,.
\]
The remaining cohomology of $X$, spanned by the primitive cohomology, is
contained in the middle dimension $H^r(X)$. Let
\[
  \chi(X) = \int_X c_r(T_X)
\]
denote the Euler characteristic of $X$.

\subsection{Quantum cohomology}

Let $X \subset \bP^{r+L}$ be a smooth Fano complete intersection of dimension $r
\geq 3$. Then the group $H_2(X,\Z)$ is freely generated by the class of a line
contained in $X$. A curve class in $X$ is therefore determined by the associated
projective degree, a non-negative integer.

The (small) quantum cohomology ring $\QH(X)$ is an algebra over the polynomial
ring $\Q[q]$ in one variable $q$. As a $\Q[q]$-module, we have
\[
  \QH(X) = H^*(X) \otimes_\Q \Q[q] \,.
\]
The product in $\QH(X)$ of two classes $\Gamma_1,
\Gamma_2 \in H^*(X)$ is defined by
\[
  \Gamma_1 \star \Gamma_2 =
  \sum_{d \geq 0} q^d (\Gamma_1 \star \Gamma_2)^X_d \,,
\]
where $(\Gamma_1 \star \Gamma_2)^X_d \in H^*(X)$ is the unique class satisfying
\[
  \int_X (\Gamma_1 \star \Gamma_2)^X_d \cdot \Gamma_3 =
  \langle\Gamma_1, \Gamma_2, \Gamma_3\rangle^X_{0,d} =
  \int_{[\Mb_{0,3}(X,d)]^\vir} \ev_1^*(\Gamma_1) \cdot
  \ev_2^*(\Gamma_2) \cdot \ev_3^*(\Gamma_3)
\]
for all $\Gamma_3 \in H^*(X)$.

Since the virtual fundamental class $[\Mb_{0,3}(X,d)]^\vir$ has
$\R$-dimension
\[
  2r + 2\int_d c_1(T_X) \,,
\]
we see that $\QH(X)$ is a graded ring where the elements of $H^*(X)$ have
their usual real degrees and
\[
  \deg(q) = 2\int_{\text{line}} c_1(T_X) = 2(r+L+1-|\mm|) \,.
\]
For $\mA \in H^*(X)$, let $\mA^k \in H^*(X)$ and $\mA^{\star k} \in \QH(X)$
denote the $k^{\mathrm{th}}$ powers with respect to the classical and quantum
products respectively.

In the following, we will assume $\deg(q) \geq \deg(\mH^2)$, or equivalently,
\begin{equation}\label{eqn:findex}
  |\mm| \leq r+L-1 \,.
\end{equation}
In the case \eqn{findex}, $\QH(X)$ satisfies the relation \cite{givental:mirror}
(see also \cite[\S 3]{pandharipande:rational} for an exposition):
\begin{equation}
  \label{eqn:magic}
  \mH^{\star(r+1)} =
  \mm^{\mm}\, q \mH^{\star(|\mm| - L)} \,.
\end{equation}
By condition \eqn{findex}, we have $0\leq |\mm| - L \leq r-1$.

We require the following result which was proved by Graber when $X$ is a
hypersurface \cite[Prop.~4]{pandharipande:rational}. The proof in the complete
intersection case is the same.

\begin{prop}[Graber]
  \label{prop:graber}
  Let $\mR = \Span \{1, \mH, \dots, \mH^r \} \subset H^*(X)$ be the
  subspace of restricted classes. Then, $(\mR \otimes_\Q \Q[q], \star)$
  is a subring of $\QH(X)$.
\end{prop}

Let $\QH(X)^\res = \mR \otimes_\Q \Q[q] \subset \QH(X)$ denote the
subring of \Proposition{graber}.

\begin{lemma}
  \label{lemma:deg2Nclass}
  Let $\Gamma \in \QH(X)^\res$ be any class of degree $2r$ satisfying
  \[
    \Gamma \equiv a \mH^r \ \text{\em mod}\ q \ \ \ \ \text{and} \ \ \ \
    \mH \star \Gamma = b\, q \mH^{\star(|\mm|-L)} \,,
  \]
  for $a,b \in \Q$. Then, $\Gamma = a \mH^{\star r} + (b - a
  \mm^{\mm})\, q \mH^{\star (|\mm|-L-1)}$.
\end{lemma}
\begin{proof}
  By the definition of the quantum product, we have
  \[
    \mH^i\equiv\mH^{\star i} \ \text{mod}\ q
  \]
  for $0\leq i \leq r$. By \Proposition{graber}, $\mH^{\star i}\in \QH(X)^\res$.
  Therefore, $\{1, \mH, \mH^{\star 2}, \mH^{\star3}, \dots, \mH^{\star r} \}$ is
  a basis of $\QH(X)^\res$ as a $\Q[q]$-module. We can write
  \[
    \Gamma = a\mH^{\star r} +\sum_{i=0}^{r-1} f_i(q) \mH^{\star i} \ \in
    \QH(X)^\res
  \]
  where $q$ divides $f_i$ for all $0\leq i \leq r-1$. Then,
  \[
    \mH\star \Gamma = a\mH^{\star (r+1)} +\sum_{i=0}^{r-1} f_i(q) \mH^{\star
    (i+1)} \ \in \QH(X)^\res \,.
  \]
  Using the second assumption and \eqn{magic}, we see
  \[
    b\, q \mH^{\star(|\mm|-L)} = a \mm^{\mm} q
    \mH^{\star(|\mm|-L)} +\sum_{i=0}^{r-1} f_i(q) \mH^{\star (i+1)}
    \ \in \QH(X)^\res \,,
  \]
  which implies the desired result by extracting the
  $\mH^{\star(|\mm|-L)}$ term.
\end{proof}

The primitive cohomology of $X$ is the linear subspace of $H^r(X)$ annihilated
by classical multiplication by $\mH$ in $H^*(X)$. The next result states that
primitive classes are also annihilated by quantum multiplication by $\mH$.

\begin{cor}\label{cor:Hmiddle}
  Let $\mA \in H^r(X)$ satisfy $\mH \cdot \mA = 0 \in H^{r+2}(X)$. Then
  \[
    \mH \star \mA = 0 \ \in \QH(X) \,.
  \]
\end{cor}
\begin{proof}
  By \Proposition{graber}, we have
  \[
    \int_X (\mH \star \mA)^X_d \cdot \mH^i =
    \langle \mH, \mA, \mH^i\rangle^X_{0,d} =
    \int_X (\mH \star \mH^i)^X_d \cdot \mA = 0
  \]
  for all $d \geq 1$ and $i \geq 0$. Since $(\mH \star \mA)^X_d$ is a class of
  degree
  \[
    2 + r - 2d(r+L+1-|\mm|) < r \,,
  \]
  the Lefschetz theorem implies that $(\mH \star \mA)^X_d = 0$.
\end{proof}

\begin{remark}
  Primitive cohomology classes are not annihilated by quantum multiplication by
  arbitrary restricted classes. For example, it follows from \Corollary{pieriq2}
  below that $H^{r+L+1-|\mm|} = H \star H^{r+L-|\mm|} - \mm!\, q$, so
  \Corollary{Hmiddle} shows that
  \[
    H^{r+L+1-|\mm|} \star \mA = - \mm!\, q \mA
  \]
  whenever $\mA \in H^r(X)$ is primitive.
\end{remark}

\begin{prop}
  \label{prop:midmid}
  Let $\mA, \mB \in H^r(X)$ satisfy $\mH \cdot \mA = \mH \cdot \mB = 0$ and $\mA
  \cdot \mB = \mP$ in $H^*(X)$. Then, we have
  \[
    \mA \star \mB = \mm^{-1} \mH^{\star r} -
    \mm^{\mm-1}\, q \mH^{\star (|\mm|-L-1)} \,.
  \]
  In particular, $\mA \star \mB \in \QH(X)^\res$.
\end{prop}
\begin{proof}
  If $\mA \star \mB \in \QH(X)^\res$, then the Proposition follows immediately
  from \Corollary{Hmiddle} and \Lemma{deg2Nclass} with $a=\mm^{-1}$ and $b=0$.

  To prove $\mA \star \mB \in \QH(X)^\res$, we proceed by contradiction. If $\mA
  \star \mB \notin \QH(X)^\res$, then there exists a primitive class $\gamma \in
  H^r(X)$ for which
  \begin{equation}\label{eqn:ll33}
    \langle \mA,\mB,\gamma \rangle^X_{0,d}\neq 0
  \end{equation}
  for some $d\geq 1$. By the dimension constraint, we have
  \[
    3r = 2r+ 2d(r+L+1-|\mm|) \,.
  \]
  In particular, $r$ is even. Since $r\geq 3$ by assumption,
  \begin{equation}\label{eqn:vrr4}
    |\mm| = \left( \frac{2d-1}{2d}\right) r + L + 1 \geq L+3 \,.
  \end{equation}
  By inequality \eqn{vrr4}, $X$ cannot be a complete intersection of two
  quadrics (since then $|\mm|=4$ and $L=2$).

  Let $V = H^r(X)_{\text{prim}} \otimes_\Q \C$ be the primitive cohomology of
  $X$ with complex coefficients. Let $G \subset \GL(V)$ denote the algebraic
  monodromy group defined as the Zariski closure of monodromy on primitive
  cohomology obtained by letting $X$ vary in the full family of nonsingular
  complete intersections of degrees $\mm$ in $\bP^{r+L}$. Since $X$ has
  dimension $r \geq 3$ and is not a complete intersection of two quadrics, $G$
  is as large as possible by results of Deligne
  \cite[Thm.~4.4.1]{deligne:conjecture*2}, see also
  \cite[Prop.~4.2]{arguz.bousseau.ea:gromov-witten}. More precisely $G = \mO(V)$
  is the full orthogonal group of $V$ with respect to the Poincar\'e duality
  pairing if $r$ is even and $G = \Sp(V)$ if $r$ is odd. By the deformation
  invariance of Gromov-Witten theory, the 3-point Gromov-Witten bracket
  \[
    \langle -,-,- \rangle^X_{0,d} : V^{\otimes 3} \to \C
  \]
  is invariant under the action of $G$. But since $-\text{Id}\in G$ acts as $-1$
  on 3-tensors, we deduce that $\langle -,-,- \rangle^X_{0,d}$ vanishes on
  $V^{\otimes 3}$, which contradicts \eqn{ll33}.
\end{proof}

\begin{remark}
  Let $\mA, \mB \in H^r(X)$ be primitive classes such that $\mA \cdot \mB = 0
  \in H^*(X)$. We can choose a primitive class $\mB' \in H^r(X)$ such that $\mA
  \cdot \mB' = \mP$. Since \Proposition{midmid} implies that $\mA \star \mB' =
  \mA \star (\mB + \mB')$, we deduce that $\mA \star \mB = 0$.
\end{remark}

\subsection{A Pieri formula modulo $q^2$}

Define the polynomial $\Psi_{\mm} \in \Z[H_1,H_2]$ by
\[
  \Psi_{\mm} = \prod_{i=1}^L \prod_{j=0}^{m_i} (j H_1 + (m_i-j) H_2) \,.
\]
The total degree of $\Psi_{\mm}$ is $L+|{\mm}|$. Our calculations
in $\QH(X)$ will use the following result for the Gromov-Witten invariants of
$X$ in the class of a line,
\[
  \langle \mH^a, \mH^b\rangle^X_{0,1} =
  \int_{[\Mb_{0,2}(X,1)]^\vir} \ev_1^*(\mH^a) \cdot
  \ev_2^*(\mH^b) \,.
\]

\begin{prop}\label{prop:gw2pt}
  Let $|\mm| \leq r+L-1$, and let $a,b \geq 0$ satisfy $a+b =
  2r+L-|\mm|$. Then, we have
  \[
    \langle \mH^a, \mH^b\rangle^X_{0,1} =
    \coeff(\Psi_\mm, H_1^{r+L-a} H_2^{r+L-b}) \,.
  \]
\end{prop}
\begin{proof}
  The Gromov-Witten invariant in question is the same for all complete
  intersections of dimension $r$ and degrees $\mm$, so we may assume that $X$ is
  general. This implies that the moduli space $\Mb_{0,2}(X,1)$ is a nonsingular
  projective variety of dimension $2r+L-|\mm|$, and the virtual fundamental
  class $[\Mb_{0,2}(X,1)]^\vir$ coincides with the usual fundamental class.

  Let $E_1, E_2 \subset \bP^{r+L}$ be general linear subspaces of codimensions
  $a$ and $b$ respectively. Using the transitive action of $\GL(r+L+1)$ on
  $\bP^{r+L}$, Kleiman's Bertini Theorem implies that $\langle H^a,
  H^b\rangle^X_{0,1}$ is equal to the number of lines contained in $X$ that meet
  both $E_1$ and $E_2$. The assumptions imply that $a + b > r$, so we have
  \[
    E_1 \cap E_2 \cap X = \emptyset\,.
  \]
  Therefore, no line in $X$ meets $E_1$ and $E_2$ in the same point. The variety
  of lines in $X$ meeting $E_1$ and $E_2$ can be identified with the set
  \[
    \{ (P_1,P_2) \in E_1 \times E_2 \mid f_i(s P_1 + t P_2) = 0
    \text{ for all $1 \leq i \leq L$ and $(s:t) \in \bP^1$} \} \,,
  \]
  where $(f_1,\ldots,f_L)$ are the defining equations for $X \subset \bP^{r+L}$.
  We write
  \[
    f_i(s P_1 + t P_2) = \sum_{j=0}^{m_i} f_{i,j}(P_1,P_2)\, s^j t^{m_i-j} \,,
  \]
  where $f_{i,j} \in H^0(E_1 \times E_2, \cO(j) \boxtimes \cO(m_i-j))$ is a
  section of the external tensor product of $\cO_{E_1}(j)$ and
  $\cO_{E_2}(m_i-j)$. We deduce that $\langle H^a,H^b\rangle^X_{0,1}$ is the
  number of points in the subscheme
  \[
    Z = V(\{f_{i,j} : 0 \leq j \leq m_i, 1 \leq i \leq L\})
    \subset E_1 \times E_2 \,.
  \]

  The standard construction \cite{fulton.pandharipande:notes} of $\barM_{0,2}(X,1)$ for a general
  complete intersection $X$ shows that the open subscheme $\ev^{-1}(X \times X
  \ssm \Delta X)$ is isomorphic to $V(\{f_{i,j}\}) \subset \bP^{r+L} \times
  \bP^{r+L} \ssm \Delta \bP^{r+L}$, so $Z$ is isomorphic to $\ev_1^{-1}(E_1)
  \cap \ev_2^{-1}(E_2) \subset \barM_{0,2}(X,1)$. In particular, $Z$ is a
  reduced complete intersection of class
  \[
    \Psi_{\mm} \in H^{2L+2|\mm|}(E_1\times E_2) \,,
  \]
  where the variables $H_1$ and $H_2$ are viewed as the generators of $\Pic(E_1
  \times E_2)$. The number of points of $Z$ is then the integral of
  $\Psi_{\mm}$ over $E_1\times E_2$.
\end{proof}

For $i \in \Z$, we define the following constant for notational convenience:
\[
  c_i = \mm^{-1} \coeff(\Psi_{\mm}, H_1^i
  H_2^{L+|\mm|-i}) \,.
\]

\begin{prop}\label{prop:ciprop}
  The constants $c_i$ satisfy the following basic properties.\smallskip

  \noin{\rm(a)} \ $c_i \neq 0$ if and only if $L \leq i \leq
  |\mm|$,\smallskip

  \noin{\rm(b)} \ $c_{L+|\mm|-i} = c_i$ for all $i \in \Z$,\smallskip

  \noin{\rm(c)} \ $c_L = c_{|\mm|} = \mm!$,\smallskip

  \noin{\rm(d)} \ $\sum_{i=L}^{|\mm|} c_i =
  \mm^{\mm}$.\smallskip
\end{prop}
\begin{proof}
  Parts (a), (b), and (c) are immediate from the definition of
  $\Psi_{\mm}$, and part (d) holds because $\Psi_{\mm}(1,1) =
  \mm^{\mm+1}$.
\end{proof}

\begin{cor}\label{cor:pieriq2}
  For $0 \leq i \leq r$, we have
  \[
    \mH \star \mH^i \ \equiv \ \mH^{i+1} +\,
    c_{i-r+|\mm|}\, q \mH^{i-r-L+|\mm|}
    \mod \ q^2 \,.
  \]
\end{cor}
\begin{proof}
  By definition of the quantum product,
  \[
    \mH \star \mH^i \ \equiv \ \mH^{i+1} +\, q\,
    \langle \mH, \mH^i, \mH^{2r-i+L-|\mm|}
    \rangle_{0,1}^X \, \mm^{-1}\,
    \mH^{i-r-L+|\mm|} \mod q^2 \,.
  \]
  By the divisor equation and the definition of $c_{r+L-i}$,
  \[
    \mH \star \mH^i \ \equiv \ \mH^{i+1} +\, c_{r+L-i}\, q \mH^{i-r-L+|\mm|}
    \mod q^2 \,,
  \]
  which is equivalent to the claim by \Proposition{ciprop}(b).
\end{proof}

\begin{cor}\label{cor:HaHbP}
  Let $a, b \geq 1$ satisfy $a+b = r+L+1-|\mm|$. Then,
  \[
    \langle \mH^a, \mH^b, \mP\rangle_{0,1}^X = \mm! \,.
  \]
\end{cor}
\begin{proof}
  We have $\mH^a \star \mH^b = \mH \star \mH^{a-1} \star \mH^b = \mH \star
  \mH^{r+L-|\mm|}$, where the last equality follows since
  \[
    a-1+b< r+L+1-|\mm| \,.
  \]
  By \Corollary{pieriq2} and \Proposition{ciprop},
  \[
    \langle \mH^a, \mH^b, \mP \rangle^X_{0,1} =
    \langle \mH, \mH^{r+L-|\mm|}, \mP \rangle^X_{0,1} =
    c_L = \mm! \,,
  \]
  as claimed.
\end{proof}

\begin{cor}\label{cor:HN}
  We have $\mH^r \equiv \mH^{\star r} +\, (\mm! -
  \mm^{\mm})\, q \mH^{\star (|\mm|-L-1)} \mod q^2$.
\end{cor}
\begin{proof}
  The result follows from \Lemma{deg2Nclass} and \Proposition{ciprop}, as we
  have
  \[
    \mH \star \mH^r \equiv c_{|\mm|}\, q \mH^{|\mm|-L} \mod q^2
  \]
  by \Corollary{pieriq2}.
\end{proof}

\subsection{Quantum Euler class}

Let $X = V(f_1,\dots,f_L) \subset \bP^{r+L}$ be a nonsingular complete
intersection of dimension $r\geq 3$ and degrees $\mm=(m_1,\ldots, m_L)$
satisfying
\[
  |\mm| \leq r + L-1 \,.
\]
Our next results determine the quantum Euler class $\mE$ of $X$
modulo $q^2$.

\begin{thm}
  \label{thm:HDeltaq2}
  We have $\mE \in \QH(X)^\res$ and
  \[
    \mH \star \mE \equiv
    (r+L+1-|\mm|)\, \mm^{\mm-1}\, q \mH^{|\mm|-L}
    \mod q^2 \,.
  \]
\end{thm}
\begin{proof}
  \Proposition{midmid} implies $\mE \in \QH(X)^\res$. Moreover, by
  \Corollary{Hmiddle}, we see
  \[
    \mH \star \mE =
    \mm^{-1} \sum_{i=0}^r \mH \star \mH^i \star \mH^{r-i} \,.
  \]
  By repeated application of \Corollary{pieriq2}, we have
  \[
    \mH^{\star i} \equiv \mH^i +
    \left( \sum_{j=L}^{i-1-r+|\mm|} c_j \right)
    q \mH^{i-r-L-1+|\mm|} \mod q^2\,.
  \]
  Expanding modulo $q^2$ yields:
  \[
    \begin{split}
      \mH \star \mH^i \star \mH^{r-i}
      &\equiv \mH^{\star (i+1)} \star \mH^{r-i} -
      \left( \sum_{j=L}^{i-1-r+|\mm|} c_j \right)
      q \mH^{|\mm|-L} \\
      &\equiv \left( \sum_{j=|\mm|-i}^{|\mm|} c_j
      - \sum_{j=L}^{i-1-r +|\mm|} c_j
      \right) q \mH^{|\mm|-L} \,.
    \end{split}
  \]
  Using \Proposition{ciprop}(b), we can rewrite the last result as
  \begin{eqnarray*}
    \left( \sum_{j=|\mm|-i}^{|\mm|} c_j - \sum_{j=L}^{i-1-r
      +|\mm|} c_j
    \right) q \mH^{|\mm|-L} &=&
    \left( \sum_{j=L}^{L+i} c_j - \sum_{j=L}^{i-1-r+|\mm|} c_j \right)
      q \mH^{|\mm|-L} \\
      & =& \left( \sum_{j=i-r+|\mm|}^{L+i} c_j
    \right) q \mH^{|\mm|-L} \,.
 \end{eqnarray*}
 Finally, we obtain by analyzing the summation:
  \[
    \begin{split}
      \mm^{-1}\sum_{i=0}^r \mH \star \mH^i \star \mH^{r-i}
      &\equiv \mm^{-1} (r+L+1-|\mm|)
      \left( \sum_{j=L}^{|\mm|} c_j \right)
      q \mH^{|\mm|-L} \mod q^2\\
      &\equiv (r+L+1-|\mm|)\, \mm^{\mm-1}\,
      q \mH^{|\mm|-L} \mod q^2\,,
    \end{split}
  \]
  where the last equality uses \Proposition{ciprop}(d).
\end{proof}

\begin{cor}
  \label{cor:Deltaq2}
  We have
  \[
    \mE \equiv \chi(X) \mm^{-1} \mH^{\star r} +
    (r+L+1-|\mm| - \chi(X))\, \mm^{\mm-1}\,
    q \mH^{\star (|\mm|-L-1)} \mod q^2 \,.
  \]
\end{cor}

In fact, the quantum Euler class is likely even better behaved. We
conjecture the following stronger result.

\begin{conj}\label{conj:delta}
  Let $X\subset \bP^{r+L}$ be a nonsingular complete intersection of dimension
  $r\geq 3$ and degrees $\mm=(m_1,\ldots, m_L)$ satisfying
  \[
    |\mm| \leq r + L-1 \,.
  \]
  Then, we have
  \[
    \mH \star \mE = (r+L+1-|\mm|)\, \mm^{\mm-1}\,
    q \mH^{\star (|\mm|-L)} \,,
  \]
  or equivalently,
  \[
    \mE = \chi(X)\, \mm^{-1}\, \mH^{\star r} +
    (r+L+1-|\mm| - \chi(X))\, \mm^{\mm-1}\,
    q \mH^{\star (|\mm|-L-1)} \,.
  \]
\end{conj}

The equivalence of the two claims in \Conjecture{delta} follows from
 \Lemma{deg2Nclass} and \eqn{magic}. \Conjecture{delta} is a consequence of
 \Theorem{HDeltaq2} when $\deg(q^2) > \deg(\mP)$, or equivalently, when
\[
  r > 2|\mm|-2L-2 \,.
\]
We have verified \Conjecture{delta} by computer in all cases where $X$ is a
complete intersection of dimension at most 30 or a hypersurface of dimension at
most 135.

\subsection{Tevelev degrees}

Since $\{ 1, \mH, \mH^{\star 2}, \mH^{\star 3}, \ldots, \mH^{\star r}\}$ is a
$\Q[q]$-module basis of $\QH(X)^\res$, we can uniquely express the point class
$\mP$ in $\QH(X)$ as
\begin{equation}
  \label{eqn:Pi}
  \mP = \sum_{i=0}^{i_0} P_i\, q^i \mH^{\star (r-i(r+L+1-|\mm|))}
\end{equation}
where $P_0, \dots, P_{i_0} \in \Q$ and
\[
  i_0=\left \lfloor\frac{r}{r+L+1-|\mm|}\right \rfloor \,.
\]
The first two coefficients
\[
  P_0 = \mm^{-1} \,, \ \ \
  P_1 = (\mm-1)! - \mm^{\mm-1}
\]
are determined by \Corollary{HN}.

\begin{defn}\label{defn:disc}
  Let $g,n \geq 0$ be non-negative integers such that
  \[
    d = \frac{r(n + g - 1)}{r + L + 1 - |\mm|}
  \]
  is a non-negative integer. There are unique rational numbers $b_i \in \Q$ such
  that
  \[
    \mP^{\star n} \star \mE^{\star g} \, = \,
    \sum_{i=0}^{i_0} b_i\, q^{d+i}\, \mH^{\star(r-i(r+L+1-|\mm|))} \,.
  \]
  Define the {\bf non-contributing part} of $\mP^{\star n} \star \mE^{\star g}$
  to be the sum
  \[
    [\mP^{\star n} \star \mE^{\star g}]^+ \, = \,
    \mP^{\star n} \star \mE^{\star g} - b_0\, q^d H^{\star r} \, = \,
    \sum_{i=1}^{i_0} b_i\, q^{d+i}\, \mH^{\star(r-i(r+L+1-|\mm|))} \,,
  \]
  and define the {\bf discrepancy} of $\mP^{\star n} \star \mE^{\star g}$ to be
  the rational number
  \[
    \Disc(\mP^{\star n} \star \mE^{\star g}) \, = \,
    \sum_{i=1}^{i_0} b_i \, \mm^{-i\mm + 1} \,.
  \]
  The product $\mP^{\star n} \star \mE^{\star g}$ is {\bf discrepancy-free} if
  $\Disc(\mP^{\star n} \star \mE^{\star g}) = 0$.
\end{defn}

In the main cases that we will consider, the non-contributing part $[\mP^{\star
n} \star \mE^{\star g}]^+$ is easy to compute and has few non-zero terms. A
formula for the virtual Tevelev degree,
\[
  \vTev^X_{g,d,n} = \coeff(\mP^{\star n} \star \mE^{\star g}, q^d \mP) \,,
\]
is implied by \Conjecture{delta}.

\begin{prop}\label{prop:tevdeg2}
  Suppose \Conjecture{delta} holds for the complete intersection
  \[
    X = V(f_1,\dots,f_L) \subset \bP^{r+L} \,,
  \]
  and let $g, d, n \geq 0$ satisfy the dimension constraint
  \begin{equation}\label{eqn:ci-constraint}
    (r+L+1-|\mm|)d = r(n+g-1) \,.
  \end{equation}
  Then, we have
  \[
    \vTev^X_{g,d,n} \ = \
    \left( \sum_{i=0}^{i_0} P_i\, \mm^{-i \mm} \right)^n
    (r+L+1-|\mm|)^g\, \mm^{d \mm - g + 1}
    \, - \, \Disc(\mP^{\star n} \star \mE^{\star g}) \,.
  \]
\end{prop}
\begin{proof}
  Using the relation \eqn{magic} and \Conjecture{delta}, we obtain
  \[
    \mP^{\star n} \star \mE^{\star g} \star \mH^{\star (|\mm|-L)}
    = \left( \sum_{i=0}^{i_0} P_i\, \mm^{-i \mm} \right)^n
    (r + L + 1 - |\mm|)^g\, \mm^{-g}\, \mH^{\star (r n + r g + |\mm| - L)} \,,
  \]
  and by \eqn{magic} and \Definition{disc} we have
  \[
    \mP^{\star n} \star \mE^{\star g} \star \mH^{\star (|\mm|-L)}
    = \sum_{i=0}^{i_0} b_i\, \mm^{-(d+i)\mm}\,
    \mH^{\star(r n + r g + |\mm| - L)} \,.
  \]
  By comparing these identities, we obtain
  \[
    \sum_{i=0}^{i_0} b_i\, \mm^{-i \mm}
    = \left( \sum_{i=0}^{i_0} P_i \mm^{-i \mm} \right)^n
    (r + L + 1 - |\mm|)^g\, \mm^{d\mm - g} \,.
  \]
  By observing that
  \[
    \vTev^X_{g,d,n} =
    \coeff(\mP^{\star n} \star \mE^{\star g}, q^d \mP) = \mm^1\, b_0 \,,
  \]
  the Proposition follows from this identity.
\end{proof}

\begin{example}
  Let $Q^r \subset \bP^{r+1}$ be a quadric of dimension $r$ and let $g,d,n \geq
  0$ satisfy $d = n+g-1$. By \Corollary{HN} and \Corollary{Deltaq2} we have
  \[
    \mP = \frac{1}{2}\, \mH^{\star r} - q
    \ \ \ \ \text{and} \ \ \ \
    \mE = \frac{r + \delta}{2}\, \mH^{\star r} - 2 \delta\, q \,.
  \]
  From this we obtain
  \[
    \begin{split}
      & \sum_{i=0}^{i_0} P_i\, \mm^{-i \mm} = \frac{1}{4} \,,
      \ \ \ \
      [\mP^{\star n} \star \mE^{\star g}]^+ = (-q)^{n+g} (2\delta)^g \,, \\
      & \text{and} \ \ \ \ \Disc(\mP^{\star n} \star \mE^{\star g}) =
      \frac{(-1)^{n+g} (2 \delta)^g}{2} \,.
    \end{split}
  \]
  \Proposition{tevdeg2} therefore gives
  \[
    \vTev^{Q^r}_{g,d,n} \, = \,
    4^{-n}\, r^g\, 2^{2d-g+1} - \Disc(\mP^{\star n} \star \mE^{\star g})
    \, = \, \frac{(2r)^g + (-1)^d\,(2\delta)^g}{2} \,.
  \]
\end{example}

\begin{remark}
  Let $g,n \geq 0$ satisfy the condition of \Definition{disc}. The product
  $\mP^{\star n} \star \mE^{\star g}$ is always discrepancy-free when $X =
  \bP^r$ is projective space, but never discrepancy-free when $X$ is a quadric.

  If \Conjecture{delta} holds for $X$, then $\mP^{\star n} \star \mE^{\star g}$
  is also discrepancy-free whenever the following inequality is satisfied:
  \begin{equation}
    \label{eqn:simple}
    n\, (r - i_0(r+L+1-|\mm|)) +
    g\, (|\mm|-L-1) \geq |\mm|-L \,.
  \end{equation}
  In fact, this inequality implies that the expansion of $\mP^{\star n} \star
  \mE^{\star g}$, using the formulas of \Conjecture{delta} and equation
  \eqn{Pi}, is a linear combination of terms $q^e \mH^{\star p}$ with $p \geq
  |\mm|-L$. Using \eqn{magic}, this implies that $b_i = 0$ for $i > 0$ in
  \Definition{disc}.

  If $X$ is not a quadric or $\bP^r$, so that $|\mm| > L+1$, then inequality
  \eqn{simple} is satisfied if either $g \geq 2$, or if $g \geq 1$, $n \geq 1$,
  and $r$ is not divisible by $r+L+1-|\mm|$.

  Computer evidence suggests that $\mP^{\star n} \star \mE^{\star g}$ is
  discrepancy-free if and only if $X \cong \bP^r$ or inequality \eqn{simple}
  holds.
\end{remark}

\subsection{Complete intersections of low degree}\label{sec:secci}

When the dimension of
\[
  X = V(f_1,\dots,f_L) \subset \bP^{r+L}
\]
is large compared to the degrees $\mm$ of the defining equations,
\[
  \deg(q^2) > \deg(\mP) \,.
\]
More precisely, the above condition holds if and only if
\begin{equation}\label{eqn:q2q2}
  r > 2|\mm|-2L-2 \,.
\end{equation}
If $X$ is not projective space, then two consequences of \eqn{q2q2} are:
\begin{enumerate}
  \item [(i)] $i_0 = 1$,
  \item [(ii)] \Conjecture{delta} holds for $X$.
\end{enumerate}
If we further assume that $X$ is not a quadric, then \eqn{simple} holds if and
only if $g+n \geq 2$.
\begin{enumerate}
    \item [(iii)] $g+n \geq 2$ implies $\Disc(\mP^{\star n} \star \mE^{\star g})
    = 0$.
\end{enumerate}
By putting our results together, we obtain the following formula for the virtual
Tevelev degrees of $X$.

\begin{thm}\label{thm:ccii}
  Let $|\mm| > L+1$ and $r > 2|\mm|-2L-2$. If $g,d,n \geq 0$ satisfy
  \[
    (r+L+1-|\mm|) d = r(n+g-1)
  \]
  and $g+n \geq 2$, then
  \[
    \vTev^X_{g,d,n} = ((\mm-1)!)^n\, (r+L+1-|\mm|)^g\,
    \mm^{(d-n)\mm-g+1} \,.
  \]
\end{thm}
\begin{proof}
  Since we have (ii) and (iii), \Proposition{tevdeg2} implies the result. We
  have $i_0 = 1$ by (i) and
  \[
    \mP = \mm^{-1}\, \mH^{\star r} + ((\mm-1)! - \mm^{\mm-1})\,
    q \mH^{\star (|\mm|-L-1)}
  \]
  by \Corollary{HN}.
\end{proof}

\begin{example}
  Let $X \subset \bP^{r+1}$ be a nonsingular cubic $r$-fold. The Tevelev degrees
  of $X$ are given by
  \begin{equation}\label{eqn:c43c}
    \vTev^X_{g,d,n} = 2^{n}\cdot (r-1)^g\cdot 3^{3d-3n-g+1}
  \end{equation}
  for $g,d,n \geq 0$ satisfying $(r-1)d = r(n+g-1)$ and $g+n \geq 2$.

  For cubic $r$-folds of dimension $r>4$, the virtual Tevelev degrees of $X$ of
  genus $g$ are enumerative for all sufficiently large $d$ by
  \cite[Thm.~11]{lian.pandharipande:enumerativity}. Formula \eqn{c43c} should
  therefore admit a derivation (for $d$ large) by classical projective geometry
  as developed in \cite{farkas.lian:linear} for the case $\bP^r$.
\end{example}

\subsection{The border case}

For further evidence, we show that \Conjecture{delta} is also true when $X =
V(f_1,\dots,f_L) \subset \bP^{r+L}$ satisfies the condition
\[
  \deg(q^2) = \deg(\mP)\,,
\]
or equivalently, when $r = 2|\mm| - 2L - 2$.

\begin{lemma}\label{lemma:bordergw}%
  Let $r = 2|\mm| - 2L - 2$. Then, we have
  \[
    \langle \mH^{r-1}, \mP\rangle^X_{0,2} =
    \frac{\mm!(\mm! + 2 c_{L+1})}{4} \,, \ \ \
    \mH^r = \mH^{\star r} + (\mm! - \mm^\mm)\,
    q \mH^{\star(|\mm|-L-1)} -
    \frac{(\mm!)^2}{2}\, q^2 \,.
  \]
\end{lemma}
\begin{proof}
  Using \Corollary{pieriq2}, we compute
  \[
    \begin{split}
      \mH^{\star (|\mm|-L-2)}
      &= \mH^{|\mm|-L-2} \,, \\
      \mH^{\star (|\mm|-L-1)}
      &= \mH^{|\mm|-L-1} + c_L\, q \, \\
      \mH^{\star (|\mm|-L)}
      &= \mH^{|\mm|-L} + (c_L + c_{L+1})\, q \mH \,.
    \end{split}
  \]
  Set $\delta = \langle H^{r-1},\mP\rangle_{0,2}^X$. We continue to
  compute
  \begin{eqnarray*}
    \mH^{\star (r-1)} &=&
    \mH^{r-1} + (\mm^{\mm} - c_L - c_{L+1})\,
    q \mH^{\star (|\mm|-L-2)} \,, \\
    \mH^{\star r} &=&
    \mH^r + c_{L+1}\, q \mH^{|\mm|-L-1} + 2\delta\, q^2
    + (\mm^{\mm} - c_L - c_{L+1})\,
    q \mH^{\star (|\mm|-L-1)} \\
    &=& \mH^r + (\mm^{\mm} - c_L)\,
    q \mH^{\star (|\mm|-L-1)} + (2\delta - c_L c_{L+1})\, q^2 \,,\\
    \mH^{\star (r+1)} &=& c_L\, q \mH^{|\mm|-L} + 2\delta\, q^2 \mH
    + (\mm^{\mm} - c_L)\, q \mH^{\star (|\mm|-L)}
    + (2\delta - c_L c_{L+1})\, q^2 H \\
    & = & \mm^{\mm}\, q \mH^{\star (|\mm|-L)}
    + (4\delta - c_L^2 - 2 c_L c_{L+1})\, q^2 \mH \,.
  \end{eqnarray*}
  Relation \eqn{magic} therefore implies
  \[
    4\delta - c_L^2 - 2 c_L c_{L+1} = 0 \,,
  \]
  which proves the identities.
\end{proof}

\begin{thm}
  \Conjecture{delta} is true when $\deg(q^2) = \deg(\mP)$.
\end{thm}
\begin{proof}
  Define constants $a_0, a_1, a_2 \in \Q$ by
  \[
    \sum_{i=0}^r \mH^i \star \mH^{r-i} =
    a_0 \mH^{\star r} + a_1\, q \mH^{\star (|\mm|-L-1)} + a_2\, q^2 \,.
  \]
  By \Proposition{midmid} and \Corollary{Deltaq2}, it suffices to show that $a_2
  = 0$.

  For $0 < i < r/2$, we have $\mH^i = \mH^{\star i}$ and
  \[
    \mH^{r-i} =
    \mH^{\star (r-i)} - \left( \sum_{j=L}^{|\mm|-i-1} c_j \right)
    q H^{\star (|\mm|-L-1-i)} \,
  \]
  by \Corollary{pieriq2}. Hence $\mH^i \star \mH^{r-i}$ does not contribute to
  $a_2$ for $0 < i < r/2$.

  We are left to consider the two terms $\mH^r \star 1$ and $1 \star \mH^r$
  together with $(\mH^{|\mm|-L-1})^{\star 2}$. The first two terms do
  contribute $-(\mm!)^2$ to $a_2$ by \Lemma{bordergw}. For the third
  term, we have
  \begin{eqnarray*}
    (\mH^{|\mm|-L-1})^{\star 2} & = &
    (\mH^{\star (|\mm|-L-1)} - \mm!\, q)^{\star 2}\\
    & = &
    \mH^{\star r} - 2 \mm!\,
    q \mH^{\star (|\mm|-L-1)} + (\mm!)^2\, q^2 \,,
  \end{eqnarray*}
  which cancels the contribution of the first two. We conclude that $a_2=0$.
\end{proof}

\begin{cor}\label{cor:border}
  Assume that $\deg(q^2) = \deg(\mP)$. Then the dimension constraint holds for
  $g, d, n \geq 0$ if and only if $d = 2(n+g-1)$. In this case we have
  \[
    \vTev^X_{g,d,n} =
    \left( \mm!\, \mm^\mm - \frac{1}{2} (\mm!)^2 \right)^n
    (|\mm|-L-1)^g\, \mm^{(2g-2)\mm - n -g +1}
    \, - \, \Disc(\mP^{\star n} \star \mE^{\star g}) \,,
  \]
  where the discrepancy is given by
  \[
    \Disc(\mP^{\star n} \star \mE^{\star g}) =
    \begin{cases}
      \left(- \frac{1}{2} \mm^{-1} (\mm!)^2 \right)^{n-1}
      \left( n\, \mm!\, \mm^{-\mm} - n - \frac{1}{2} \mm^{-2\mm} (\mm!)^2
      \right)
      & \text{if $g=0$,} \\
      \left(- \frac{1}{2} \mm^{-1} (\mm!)^2 \right)^n
      \big(|\mm|-L-1 - \chi(X)\big)
      & \text{if $g=1$,} \\
      0 & \text{if $g \geq 2$.}
    \end{cases}
  \]
\end{cor}
\begin{proof}
  Using the formulas of \Lemma{bordergw} and \Conjecture{delta}, we obtain
  \[
    \begin{split}
    \sum_{i=0}^{i_0} P_i\, \mm^{-i \mm} &=
    \mm^{-2\mm - 1} \left( \mm!\, \mm^\mm - \frac{1}{2} (\mm!)^2 \right) \,,
    \\
    [\mP^{\star n}]^+ &=
    \left(- \frac{(\mm!)^2}{2 \mm^1}\, q^2 \right)^{n-1}
    \left( \frac{n\, (\mm! - \mm^\mm)}{\mm^1}\, q \mH^{\star(|\mm|-L-1)} -
    \frac{(\mm!)^2}{2 \mm^1}\, q^2 \right) \,\text{, and}
    \\
    [\mP^{\star n} \star \mE]^+ &=
    \left(- \frac{(\mm!)^2}{2 \mm^1}\, q^2 \right)^n
    \big(|\mm|-L-1-\chi(X)\big)\, \mm^{\mm-1}\, q \mH^{\star(|\mm|-L-1)} \,.
    \end{split}
  \]
  The Corollary therefore follows from \Proposition{tevdeg2}.
\end{proof}

\begin{example}\label{example:vTev_neg}
  Let $X \subset \bP^7$ be the intersection of three general quadrics, so
  $r=4$ and $\mm=(2,2,2)$. For $g=n=1$ and $d=2$, \Corollary{border} gives
  \[
    \vTev^X_{1,2,1} = 120 - 184 = -64 \,.
  \]
  In particular, virtual Tevelev degrees can be negative.
\end{example}

By \Proposition{flagINT}, all virtual Tevelev degrees for flag varieties are
integers. It would be interesting to know if all virtual Tevelev degrees are
integers for all varieties  $X$. We have not seen a counterexample when $X$ is a
complete intersection satisfying $|\mm| \leq r+L-1$.

\ifdefined\mybibfile
\bibliography{\mybibfile}

\begin{thebibliography}{BCMP21}

\bibitem[ABPZ21]{arguz.bousseau.ea:gromov-witten}
H.~Arg\"uz, P.~Bousseau, R.~Pandharipande, and D.~Zvonkine.
\newblock Gromov-{W}itten theory of complete intersections.
\newblock arXiv:2109.13323, 2021.

\bibitem[Abr00]{abrams:quantum}
L.~Abrams.
\newblock The quantum {E}uler class and the quantum cohomology of the
  {G}rassmannians.
\newblock {\em Israel J. Math.}, 117:335--352, 2000.

\bibitem[BCMP21]{buch.chaput.ea:positivity}
A.~S. Buch, P.-E. Chaput, L.~C. Mihalcea, and N.~Perrin.
\newblock Positivity of minuscule quantum {$K$}-theory.
\newblock in preparation, 2021.

\bibitem[Beh97]{behrend:gromov-witten}
K.~Behrend.
\newblock Gromov-{W}itten invariants in algebraic geometry.
\newblock {\em Invent. Math.}, 127(3):601--617, 1997.

\bibitem[Beh99]{behrend:product}
K.~Behrend.
\newblock The product formula for {G}romov-{W}itten invariants.
\newblock {\em J. Algebraic Geom.}, 8(3):529--541, 1999.

\bibitem[Ber97]{bertram:quantum}
A.~Bertram.
\newblock Quantum {S}chubert calculus.
\newblock {\em Adv. Math.}, 128(2):289--305, 1997.

\bibitem[BKT03]{buch.kresch.ea:gromov-witten}
A.~S. Buch, A.~Kresch, and H.~Tamvakis.
\newblock Gromov-{W}itten invariants on {G}rassmannians.
\newblock {\em J. Amer. Math. Soc.}, 16(4):901--915, 2003.

\bibitem[Buc]{buch:equivariant}
A.~S. Buch.
\newblock Equivariant {S}chubert {C}alculator, a {M}aple package for
  computations in the equivariant cohomology and {$K$}-theory of flag
  manifolds.
\newblock Available at {\tt https://math.rutgers.edu/$\sim$asbuch/equivcalc/}.

\bibitem[Buc03]{buch:quantum}
A.~S. Buch.
\newblock Quantum cohomology of {G}rassmannians.
\newblock {\em Compositio Math.}, 137(2):227--235, 2003.

\bibitem[CL21]{cela.lian:generalized}
A.~Cela and C.~Lian.
\newblock Generalized {T}evelev degrees of $\mathbb{P}^1$.
\newblock arXiv:2111.05880, 2021.

\bibitem[CMP07]{chaput.manivel.ea:quantum}
P.-E. Chaput, L.~Manivel, and N.~Perrin.
\newblock Quantum cohomology of minuscule homogeneous spaces. {II}. {H}idden
  symmetries.
\newblock {\em Int. Math. Res. Not. IMRN}, (22):Art. ID rnm107, 29, 2007.

\bibitem[CMP08]{chaput.manivel.ea:quantum*1}
P.-E. Chaput, L.~Manivel, and N.~Perrin.
\newblock Quantum cohomology of minuscule homogeneous spaces.
\newblock {\em Transform. Groups}, 13(1):47--89, 2008.

\bibitem[CMP09]{chaput.manivel.ea:affine}
P.-E. Chaput, L.~Manivel, and N.~Perrin.
\newblock Affine symmetries of the equivariant quantum cohomology ring of
  rational homogeneous spaces.
\newblock {\em Math. Res. Lett.}, 16(1):7--21, 2009.

\bibitem[CMP10]{chaput.manivel.ea:quantum*2}
P.-E. Chaput, L.~Manivel, and N.~Perrin.
\newblock Quantum cohomology of minuscule homogeneous spaces {III}.
  {S}emi-simplicity and consequences.
\newblock {\em Canad. J. Math.}, 62(6):1246--1263, 2010.

\bibitem[CPS21]{cela.pandharipande.ea:tevelev}
A.~Cela, R.~Pandharipande, and J.~Schmitt.
\newblock Tevelev degrees and {H}urwitz moduli spaces.
\newblock {\em Math. Proc. Cambridge Philos. Soc.} (to appear),
  arXiv:2103.14055, 2021.

\bibitem[Del80]{deligne:conjecture*2}
P.~Deligne.
\newblock La conjecture de {W}eil. {II}.
\newblock {\em Inst. Hautes \'Etudes Sci. Publ. Math.}, (52):137--252, 1980.

\bibitem[FL21]{farkas.lian:linear}
G.~Farkas and C.~Lian.
\newblock Linear series on general curves with prescribed incidence conditions.
\newblock arXiv:2105.09340, 2021.

\bibitem[FP97]{fulton.pandharipande:notes}
W.~Fulton and R.~Pandharipande.
\newblock Notes on stable maps and quantum cohomology.
\newblock In {\em Algebraic geometry---{S}anta {C}ruz 1995}, volume~62 of {\em
  Proc. Sympos. Pure Math.}, pages 45--96. Amer. Math. Soc., Providence, RI,
  1997.

\bibitem[FW04]{fulton.woodward:quantum}
W.~Fulton and C.~Woodward.
\newblock On the quantum product of {S}chubert classes.
\newblock {\em J. Algebraic Geom.}, 13(4):641--661, 2004.

\bibitem[Giv98]{givental:mirror}
A.~Givental.
\newblock A mirror theorem for toric complete intersections.
\newblock In {\em Topological field theory, primitive forms and related topics
  ({K}yoto, 1996)}, volume 160 of {\em Progr. Math.}, pages 141--175.
  Birkh\"auser Boston, Boston, MA, 1998.

\bibitem[Hu15]{hu:big}
X.~Hu.
\newblock Big quantum cohomology of {F}ano complete intersections.
\newblock arXiv:1501.03683, 2015.

\bibitem[KT03]{kresch.tamvakis:quantum}
A.~Kresch and H.~Tamvakis.
\newblock Quantum cohomology of the {L}agrangian {G}rassmannian.
\newblock {\em J. Algebraic Geom.}, 12(4):777--810, 2003.

\bibitem[KT04]{kresch.tamvakis:quantum*1}
A.~Kresch and H.~Tamvakis.
\newblock Quantum cohomology of orthogonal {G}rassmannians.
\newblock {\em Compos. Math.}, 140(2):482--500, 2004.

\bibitem[LP21]{lian.pandharipande:enumerativity}
C.~Lian and R.~Pandharipande.
\newblock Enumerativity of virtual {T}evelev degrees.
\newblock arXiv:2110.05520, 2021.

\bibitem[Pan98]{pandharipande:rational}
R.~Pandharipande.
\newblock Rational curves on hypersurfaces (after {A}. {G}ivental).
\newblock {\em Ast\'erisque}, (252):Exp.\ No.\ 848, 5, 307--340, 1998.
\newblock S{\'e}minaire Bourbaki. Vol. 1997/98.

\bibitem[Pos05]{postnikov:affine}
A.~Postnikov.
\newblock Affine approach to quantum {S}chubert calculus.
\newblock {\em Duke Math. J.}, 128(3):473--509, 2005.

\bibitem[Ste96]{stembridge:fully}
J.~R. Stembridge.
\newblock On the fully commutative elements of {C}oxeter groups.
\newblock {\em J. Algebraic Combin.}, 5(4):353--385, 1996.

\bibitem[Tev20]{tevelev:scattering}
J.~Tevelev.
\newblock Scattering amplitudes of stable curves.
\newblock arXiv:2007.03831, 2020.

\bibitem[Wit95]{witten:verlinde}
E.~Witten.
\newblock The {V}erlinde algebra and the cohomology of the {G}rassmannian.
\newblock In {\em Geometry, topology, \& physics}, Conf. Proc. Lecture Notes
  Geom. Topology, IV, pages 357--422. Int. Press, Cambridge, MA, 1995.

\end{thebibliography}
\else

\fi
\bibliographystyle{halpha}

\end{document}